\documentclass[11pt]{amsart}
\usepackage{amsmath,amssymb,latexsym}
\usepackage[T1]{fontenc}
\usepackage{ upgreek }
\usepackage{txfonts}
\usepackage[all]{xy}
\usepackage{mathrsfs}
\usepackage{amsfonts}
\usepackage{xcolor}
\usepackage{txfonts}
\usepackage{dsfont}
\usepackage{enumerate}
\usepackage[normalem]{ulem} 
\usepackage{mathptmx}
\usepackage[scaled=.90]{helvet}
\usepackage{courier}
\usepackage{hyphenat} 
\usepackage[bookmarksnumbered]{hyperref}
\newtheorem{theorem}{Theorem}

\newtheorem{prop}[theorem]{Proposition}
\newtheorem{corollary}[theorem]{Corollary}
\newtheorem{lemma}[theorem]{Lemma}
\newtheorem{defi}{Definition}
\newtheorem{remark}[theorem]{Remark}

\newcommand{\R}{\mathds R}

\newcommand{\II}{\mathrm I\!\mathrm I_v}
\newcommand{\IIs}{\mathrm I\!\mathrm I}
\newcommand{\WSigP}{W^{\top_v}_{S\cap \Sigma}}
\newcommand{\WSigPV}{W^{\top_V}_{S\cap \Sigma}}
\newcommand{\V}{\mathcal V}
\newcommand{\lW}{\ell(W)}

\def\bb#1\eb{\textcolor{blue}
	{#1}} %
\def\bm#1\em{\textcolor{magenta}
	{#1}} %
	\def\br#1\er{\textcolor{red}
	{#1}} %

\def\nvm{\widetilde{\nabla}^V}
\def\nvi{\widehat{\nabla}^V}

\def\V{\mathscr{V}}

\def\h{{h}}

\def\<{\langle}

\def\>{\rangle}

\title[The flag curvature of a submanifold of a Randers-Minkowski space]{The flag curvature of a submanifold of a Randers-Minkowski space in terms of Zermelo data}

\author[M. Huber]{Matthieu Huber}
\address{Departamento de Matem\'aticas, \hfill\break\indent
Universidad de Murcia, \hfill\break\indent
Campus de Espinardo,\hfill\break\indent
30100 Espinardo, Murcia, Spain}
\email{matthieu.huber@um.es}

\author[M. A. Javaloyes]{Miguel Angel Javaloyes}
\address{Departamento de Matem\'aticas, \hfill\break\indent
Universidad de Murcia, \hfill\break\indent
Campus de Espinardo,\hfill\break\indent
30100 Espinardo, Murcia, Spain}
\email{majava@um.es}

\thanks{ This work is a result of the activity developed within the framework of the Programme in
 	Support of Excellence Groups of the Regi\'on de Murcia, Spain, by Fundaci\'on S\'eneca, Science and Technology Agency of the Regi\'on de Murcia. The second author was partially supported by MICINN/FEDER project with reference PGC2018-097046-B-I00  and Fundaci\'on S\'eneca project with reference 19901/GERM/15. }
 	

\begin{document}

\begin{abstract}
The main result of this paper is an expression of the flag curvature of a submanifold of a Randers-Minkowski space $(\V,F)$ in terms of invariants related to its Zermelo data $(h,W)$. More precisely, these invariants are the sectional curvature and
the second fundamental form of the positive definite scalar product $h$ and some projections of the wind $W$. This expression allows for a promising characterization of submanifolds with scalar flag curvature in terms of Riemannian quantities, which, when a hypersurface is considered, seems quite approachable. As a consequence, we prove that any $h$-flat hypersurface $S$ has scalar $F$-flag curvature  and the metric of its Zermelo data is conformally flat. As a tool for making the computation, we previously reobtain the Gauss-Codazzi equations of a pseudo-Finsler submanifold using anisotropic calculus.
\end{abstract}

\maketitle
\section{Introduction}
The theory of submanifolds in Finsler Geometry has a long history beginning with the work of M. Haimovici \cite{Hai35} and J. M. Wegener \cite{Weg36}. But the complexity of its computations has slowed down its development compared with the theory of Riemannian submanifolds. 
There are several aspects of Finsler submanifolds that can deserve some attention.
The study of minimal submanifolds and mean curvature
was brought up in \cite{Sh98}. As the definition of minimal submanifold depends on the choice of a volume form, there has been some controversy as there are several possibilities for this volume form \cite{AlBe06,Be09}. A second aspect that has attracted the attention of the Finslerian community is the study of the intrinsic geometry of a submanifold. At this point, the difficulty of making computations has been a hard obstacle to overcome, but the reward is not less appealing as one can generate many examples of Finsler manifolds in a natural way,  not to mention possible applications. 

The main goal of this paper is to obtain an expression for the flag curvature of a submanifold in a Randers-Minkowski space in terms of its Zermelo data. Recall that the flag curvature is one of the most important geometric invariants of a Finsler manifold. It is  a measure of how geodesics get apart when the initial directions lie in a certain plane.  More specifically, the flagpole is the direction of the initial geodesic, so that the flag curvature depends on a plane (the flag) and on a direction (the flagpole).  On the other hand, a Randers-Minkowski space is a vector space $\V$ endowed with a Randers norm. A Randers norm $F:\V\rightarrow \R$ can be characterized by having  as indicatrix $\Sigma=F^{-1}(1)$ an ellipsoid which contains the origin in its bounded domain.
This ellipsoid determines a positive-definite scalar product $h$ and a vector $W\in\V$ in the sense that the ellipsoid coincides with the unit sphere of $h$ up to a translation $-W$ of its center of mass to the origin.
The pair $(h,W)$ is called the Zermelo data of the Randers norm as it is related to the problem of Zermelo navigation, which consists of determining the path that minimizes the travel time for a ship or an airplane in the presence of a wind. The Zermelo data was used in \cite{BRS04} to classify Randers manifolds with constant flag curvature. This was very surprising since the problem of classifying constant flag curvature Randers manifolds appeared to be very difficult when making the computations in terms of  $(\alpha,\beta)$ as in \eqref{alphabet} rather than in terms of the Zermelo data $(h,W)$. Recall that a Finsler manifold is said to be of scalar flag curvature if the flag curvature depends only on the flagpole and not on the plane containing this flagpole, namely, the flag curvature is a positive homogeneous function of degree zero on the tangent bundle. The equations characterizing Randers manifolds of scalar flag curvature obtained in \cite{YaSh77} turned out to be incomplete \cite{Sh02,Sh03} (see also \cite{Bao04} for a detailed story).   Following  in the wake of the pioneering paper \cite{BRS04}, we have computed the flag curvature of an arbitrary submanifold $S$ of a Randers-Minkowski space $(\V,F)$ using only the geometric invariants of the Zermelo data $(h,W)$. More precisely, the flag curvature of $(S,F|_S)$, where $F|_S$ is the Finsler metric on $S$ induced by the Randers-Minkowski norm $F$, depends on the $h$-Riemannian sectional curvature $K^h$, the second fundamental form $\IIs'$ (computed with $h$), and some $h$-projections of $W$ (see \eqref{Kflagformula}).

 One of the most appealing consequences of our computation of the flag curvature is that it allows us to give a characterization of submanifolds with scalar flag curvature in a purely Riemannian way.  This leads to  a huge family of examples beginning with the class described  in Corollary \ref{flatclass}, which claims that every $h$-flat hypersurface is of scalar flag curvature for $F$ and the metric of its Zermelo data is conformally flat. Indeed, the classification of hypersurfaces of scalar flag curvature seems achievable looking at the condition given in Corollary~\ref{hyperclassi}. The case of arbitrary codimension could be much more difficult to solve. For example, the classification of submanifolds with constant sectional curvature of the Euclidean space, which is a particular case of our general problem putting $W=0$, is still an open problem with some partial results in low codimension \cite[Chapter 5]{DaTo19}. In any case, the codimension is very important, since the classification of submanifolds of scalar flag curvature up to a certain codimension could be of great help to obtain the classification of Randers manifolds of scalar flag curvature, which has remained as a major open problem in Finsler Geometry (see \cite[\S 4.3]{ChengSh12} for a summary of the most important results and \cite{ShCi12,ShXia13,CheYu14} for some recent results). Recall that there are several extensions of the  Nash embedding theorem for Riemannian manifolds to the Finslerian realm  (see \cite{BuIv93,Gu56,Gu58,Ing54} and also the introduction of \cite{Sh98} for a summary and some counterexamples for smooth and global embeddings).  These extensions claim that a Finsler manifold can be isometrically embedded in some Minkowski space. Though these theorems do not specify the type of Minkowski space in which the manifold is embedded, it is quite natural to think that, if one considers a Randers manifold, the embedding can be made into a Randers space. If this could be done so much as locally, the classification of scalar flag curvature submanifolds of a Randers-Minkowski space could lead to the classification of Randers manifolds of scalar flag curvature.


  All the computations have been carried out using the anisotropic calculus developed in \cite{Jav16,Jav19}. The main contribution of this kind of calculus is that it allows us to perform all calculations using affine connections on the manifold. This is done by fixing a vector field $V$ which is an arbitrary extension of the vector $v$ where we want to make the computation. In this way, one gets an affine connection $\nabla^V$ from the given anisotropic connection (see Definition \ref{aniconnection}  and recall that, roughly speaking, an anisotropic connection  is a connection which has a different value for every direction $v$).  In order to make all the results independent from the choice of $V$, it is necessary to add a term that depends on vertical derivatives (derivatives in the tangent space) and on $\nabla^V_XV$ (see \eqref{nablaTV}). This method allows us to avoid the use of coordinates in all our computations, giving always tensorial expressions. The first one to use this family of affine connections was H.-H. Matthias \cite{Mat80} using a geodesic vector field $V$ and it was generalized by the second author  sucessively in \cite{Jav14a,Ja14b,Jav16,Jav19}. 

Let us describe with more detail the structure and results of the paper. After giving the basic notions of anistropic calculus in \S 2, the next section is devoted  to obtain a sort of Gauss-Codazzi equations for pseudo-Finsler manifolds. This is a crucial step to study the intrinsic and extrinsic geometry of a submanifold and it has been achieved several times in literature independently and
with different purposes, methods and connections (see 
 \cite{Hombu36,Da45,Ak68,Mat85,Dra86,Nie04,Li2010,Li2011,Li2012,Wu16} and also  \cite[Theorem 2.1]{BeFa00}). Our deduction of the Gauss-Codazzi equations  will use the anisotropic calculus and the Chern connection (see Theorem \ref{J}). Indeed, we will obtain the Gauss-Codazzi equations not only in terms of the second fundamental form $\IIs$, but also  in terms of the tensor $Q$, the difference between the induced connection $\widehat\nabla$ and the Chern connection $\nabla$ of the submanifold $S$. It turns out that $Q$ depends on the Cartan tensor $C$ and the second fundamental form $\IIs$ (see Lemma \ref{CartanQ}). Then we study the properties of $Q$ (see Lemma \ref{CartanQCor}) to obtain an expression of the flag curvature of a non-degenerate submanifold (see Corollary \ref{Curflag}) from the Gauss equation \eqref{GEQ}. The properties of $Q$ are crucial to cancel many terms of \eqref{GEQ}. We also recover in Corollary \ref{CurTotGeo} a result by Li \cite{Li2012}, which claims that the flag curvature of a totally geodesic submanifold (see \S 3.3 for definitions and characterizations) coincides with the flag curvature of the background manifold. With the formula for the flag curvature of a submanifold at hand, in \S 4 we obtain an explicit expression 
  of the flag curvature of a submanifold of a Randers-Minkowski space in terms of its Zermelo data $(h,W)$  (for previous results on Randers and Kropina submanifolds see \cite[\S 8]{Mat85} and \cite[\S 3.3]{BeFa00}, and for the more general class of $(\alpha,\beta)$-metrics see  \cite{Bal07}).  First we compute the fundamental tensor $g$ (see Lemma \ref{gVProptoSigmaN}) and the Cartan tensor $C$ (see Proposition \ref{CartangSigma}) in terms of the Zermelo data. Then in \S 4.1, we describe the Zermelo data of a submanifold (Proposition \ref{ZermSub})), the second fundamental form of $S$ (Lemma \ref{IIvII}) and, finally, in Lemma \ref{tedious02}, we compute the most difficult term in the expression of the flag curvature, which depends on $\nabla Q$. With all this information, we are able to obtain the formula for the flag curvature in Theorem \ref{maintheorem}, which is given only in terms of the Zermelo data $(h,W)$, namely, the $h$-sectional curvature $K^h$ of $S$, its second fundamental form $\IIs'$  with respect to $h$ and some $h$-projections of $W$. In the case of a hypersurface, the computation of the second fundamental form becomes easier (see Lemma \ref{hyperSSF}), and this allows us to give a simpler expression of the flag curvature in \eqref{Kflagformula2} and a characterization of hypersurfaces with scalar flag curvature (see Corollary \ref{hyperclassi}). This leads us to  find the mentioned family of $h$-flat hypersurfaces with scalar flag curvature and conformally flat  metric of its Zermelo data. We finish the paper giving  an explicit computation of the flag curvature of the indicatrix of a Randers-Minkowski space in Corollary \ref{indicatrix}, and making some considerations about pseudo-Randers-Kropina metrics in the final Remark \ref{pseudoRanders}. In particular, all the results extend to this more general case, including Kropina metrics.
\section{Preliminaries and notations}
Given a manifold $M$, we will denote by $\pi: TM\rightarrow M$ the natural projection from its tangent bundle $TM$.
\begin{defi} \label{gCdef}
Let  $M$ be a manifold, and $A\subset TM\setminus\{ 0\}$ be a conic open subset of $TM$, namely, if $v\in A$ then $\lambda v\in A$ for every $\lambda>0$. We say that a smooth function $L:A\longrightarrow \R$ is a pseudo-Finsler metric when it satisfies the following properties:
\begin{itemize}
\item[$i)$] $L(\lambda v)=\lambda^2L(v) $ for all $ \lambda >0$ and $v\in A$,
\item[$ii)$] the fundamental tensor $g_v:T_{\pi(v)}M\times T_{\pi(v)}M\rightarrow \R$, defined for all $v \in A$ and $u,w \in T_{\pi (v)}M$ as
\begin{equation}\label{fundtensor}
g_v(u,w):=\frac{1}{2}\frac{\partial^2}{\partial s \partial t}L(v+su+tw)\Big|_{s=t=0},
\end{equation}
 is non-degenerate.
\end{itemize}
Furthermore, we define for every $u,w,z \in T_{\pi (v)}M$ the Cartan tensor of $L$ as
\[C_v(u,w,z):=\frac{1}{2}\frac{\partial}{\partial t}g_{v+tz}(u,w)\Big|_{t=0}=\frac{1}{4}\frac{\partial^3}{\partial r \partial s \partial t}L(v+ru+sw+tz)\Big|_{r=s=t=0}.\]
\end{defi}

In the following, the space of $A$-anisotropic tensors of type $(r,s)$ will be denoted by ${\mathfrak T}^r_s(M,A)$, while by convention ${\mathfrak T}^0_0(M,A)\equiv {\mathcal F}(A)$, where ${\mathcal F}(A)$ denotes the space of smooth real functions on $A$ (see \cite{Jav16} and \cite{Jav19}). One can think of them as a generalization of classical tensors, in the sense that there is a multilinear map for every $v\in A$ rather than every $p\in M$, namely, the coordinates of the tensor are functions on an open subset of $A$. We will denote by $\mathfrak{X}(M)$ the space of smooth vector fields on $M$ and by ${\mathcal F}(M)$ the space of smooth real functions on $M$.  Observe that the elements of  ${\mathfrak T}^1_0(M,A)$ can be interpreted as $A$-anisotropic vector fields, namely, for every $v\in A$ we choose a vector $X_v\in T_{\pi(v)}M$ in a smooth way. The subset $\mathfrak{X}(M)$ can be seen as a subset of ${\mathfrak T}^1_0(M,A)$, since in the case of  a classical vector field $X$, we can define $X_v:=X_{\pi(v)}$.  
In many cases, it will be clear  what the subset $A$ is, so we wil speak just about anisotropic tensors. 
\begin{defi}\label{aniconnection}
	An  {\em $A$-anisotropic (linear) connection} is  a map
	\[\nabla: A\times \mathfrak{X}(M)\times\mathfrak{X}(M)\rightarrow TM,\quad\quad (v,X,Y)\mapsto\nabla^v_XY\in T_{\pi(v)}M,\]
	such that for any $X,Y,Z\in  {\mathfrak X}(M)$, $v\in  A$, $f,h\in {\mathcal F}(M)$, 
	\begin{enumerate}[(i)]
		\item $\nabla^v_X(Y+Z)=\nabla^v_XY+\nabla^v_XZ$, 
		\item $\nabla^v_X(fY)=X(f) Y|_{\pi(v)}+f(\pi(v)) \nabla^v_XY$,
		\item $ \nabla_XY $ is an $ A $-anisotropic vector field (as a map $A \ni v \mapsto \nabla^v_XY \in T_{\pi(v)}M $),
		\item $\nabla^v_{fX+hY}Z=f(\pi(v))\nabla^v_XZ+h(\pi(v)) \nabla^v_YZ$.
	\end{enumerate}
\end{defi}
Given a vector field $V$ on an open subset $\Omega\subset M$ which is $A$-admissible, namely, $V_p\in A$ for any $p\in \Omega$, and an $A$-anisotropic tensor $T\in {\mathfrak T}^r_s(M,A)$, we will denote by $T_V\in {\mathfrak T}^r_s(\Omega)$ the classical $(r,s)$-tensor given at every point $p\in \Omega$ by $(T_V)_p:=T_{V_p}$. We also will define the affine connection $\nabla^V$ given by $(\nabla^V_XY)_p:=\nabla^{V_p}_XY$ for any $X,Y\in \mathfrak{X}(\Omega)$.
Having at our disposal the anisotropic connection and a vector field $X\in {\mathfrak X}(M)$, one can define now an anisotropic tensor derivation $\nabla_X$  (see \cite[\S 2.2]{Jav16} for details) in the space of anisotropic tensors ${\mathfrak T}^r_s(M,A)$.   If $h\in{\mathcal F}(A)\equiv {\mathfrak T}^0_0(M,A)$, then $\nabla_Xh\in{\mathcal F}(A)$ is determined by
 \begin{equation}\label{derh}
(\nabla_X h) \circ V = X (h\circ V) - (\partial^\nu h)_V(\nabla^V_XV),
 \end{equation}
 where $V$ is any $A$-admissible vector field on $\Omega\subset M$ which extends $v$, namely, $V_{\pi(v)}=v$ and $\partial^\nu h$ is the vertical derivative of $h$ 
 defined as 
 \[(\partial^\nu h)_v(z)=\frac{\partial}{\partial t}h(v+tz)|_{t=0}\] for every $v\in A$ and $z\in T_{\pi(v)}M$. Moreover, by the chain rule, we deduce the following formula for the tensor derivative $\nabla_XT$ of 
 $T\in  {\mathfrak T}^0_s(M,A)$,
\begin{equation}\label{nablaXT}
\nabla_XT(X_1,\ldots,X_s)=\nabla_X(T(X_1,\ldots,X_s))-\sum_{i=1}^rT(X_1,\ldots,\nabla_XX_i,\ldots,X_s),
\end{equation}
where $X,X_1,\ldots,X_s\in \mathfrak{X}(\Omega)$.  Observe that, by the ${\mathcal F}(A)$-multilinearity, it is enough to define the anisotropic tensor for classical vector fields since this determines its value in arbitrary anisotropic vector fields.  In addition,
\begin{equation}\label{partialX}
\nabla_X(T(X_1,\ldots,X_s))\circ V=X(T_V(X_1,\ldots,X_s))-(\partial^\nu T)_V(X_1,\ldots,X_s,\nabla^V_XV),
\end{equation}
where the vertical derivative $(\partial^\nu T)$ is given by
\begin{equation}\label{partialT}
(\partial^\nu T)_v(Y_1,\ldots,Y_{s+1})=\frac{\partial}{\partial t} T_{v+t Y_{s+1}|_{\pi(v)}}(Y_1,\ldots,Y_s)
\end{equation}
 for $Y_1,\ldots,Y_{s+1}\in \mathfrak{X}(M)$ (see \cite[Eq. (6)]{Jav19}).
   Putting together  \eqref{nablaXT} and \eqref{partialX}, we obtain 
   a formula for $\nabla_XT$ using an $A$-admissible extension $V$ of $v$,
\begin{multline}\label{nablaTV}
(\nabla_XT)_V(X_1,\ldots,X_s)=X(T_V(X_1,\ldots,X_s))-(\partial^\nu T)_V(X_1,\ldots,X_s,\nabla^V_XV)\\-\sum_{i=1}^sT_V(X_1,\ldots,\nabla^V_XX_i,\ldots,X_s).
\end{multline}   
A similar formula can be obtained when one considers an anisotropic tensor $T\in {\mathfrak T}^1_s(M,A)$ as an ${\mathcal F}(A)$-multilinear map 
$T: {\mathfrak T}^1_0(M,A)\times \ldots\times {\mathfrak T}^1_0(M,A)\rightarrow {\mathfrak T}^1_0(M,A)$. Taking into account that $\mathfrak{X}(M)\subset {\mathfrak T}^1_0(M,A)$,
\begin{multline}\label{nablaTV2}
(\nabla_XT)_V(X_1,\ldots,X_s)=\nabla^V_X(T_V(X_1,\ldots,X_s))-(\partial^\nu T)_V(X_1,\ldots,X_s,\nabla^V_XV)\\-\sum_{i=1}^sT_V(X_1,\ldots,\nabla^V_XX_i,\ldots,X_s),
\end{multline}   
where $X,X_1,\ldots,X_s\in \mathfrak{X}(\Omega)$ and $\partial^\nu T$ is defined formally as in \eqref{partialT}.
 Eq. \eqref{derh} also allows us to extend the anisotropic derivation to ${\mathfrak T}^1_0(M,A)\times {\mathfrak T}^1_0(M,A)$ (see \cite[Remark 2.3]{Jav19}).
    Then it is possible to define its associated curvature tensor $R_v:{\mathfrak X}(M)\times {\mathfrak X}(M)\times {\mathfrak X}(M)\rightarrow T_{\pi(v)}M$:
\begin{equation}\label{Rv}
R_v(X,Y)Z=\nabla^v_X(\nabla_YZ)-\nabla^v_Y(\nabla_XZ)-
\nabla^v_{[X,Y]}Z,
\end{equation}
for any $v\in A$ and $X,Y,Z\in {\mathfrak X}(M)$. 
Moreover, the curvature tensor of $\nabla$ can be computed using $\nabla^V$:
\begin{equation}\label{RvRV}
R_v(X,Y)Z=R^V_{\pi(v)} (X,Y)Z-P_v(Y,Z,\nabla^V_XV)+P_v(X,Z,\nabla^V_YV),
\end{equation}
for any $v\in A$, where $V,X,Y,Z\in {\mathfrak{X}}(\Omega)$, being $V$ an $A$-admissible extension of $v$, $R^V$, the curvature tensor of the affine connection $\nabla^V$ and $P$ the vertical derivative of $\nabla$, namely,
\begin{align}
\label{tensorP} P_v(X,Y,Z)&=\left.\frac{\partial}{\partial t}\nabla^{v+tZ|_{\pi(v)}}_XY\right\vert_{t=0},\\
\label{CurvRV}R^V(X,Y)Z&=\nabla^V_X\nabla^V_YZ-\nabla^V_Y\nabla^V_XZ-
\nabla^V_{[X,Y]}Z.
\end{align}
Observe that $P$ is an $A$-anisotropic tensor, 
but $R^V$ is not, since it depends on the extension $V$. Moreover, for every $v\in A$, it is always possible to choose an $A$-admissible extension $V$ of $v=V|_{\pi(v)}$ on some open subset $\Omega$ such that
\begin{equation}\label{nablaXV=0}
\nabla^v_XV=0,
\end{equation}
for every $X\in \mathfrak{X}(\Omega)$ (see \cite[Proposition 2.13]{Jav19}). In the following we will express this condition as $\nabla^vV=0$.
\begin{defi}
Given  a pseudo-Finsler manifold $(M,L)$, the {\em Levi-Civita-Chern connection} is the unique anisotropic connection $\widetilde\nabla$ that is torsion-free, namely,
\begin{equation}
\label{tor}\widetilde\nabla^v_XY-\widetilde\nabla^v_YX=[X,Y], \quad \forall X,Y\in  {\mathfrak X}(M), \,\, v\in A
\end{equation}
and  is compatible with $g$, namely, $\widetilde\nabla g=0$. The last condition is equivalent to the following: for every
 $A$-admissible vector field $V$ on an open subset $\Omega \subset M$, the affine connection $\widetilde\nabla^V$ on $\Omega$ is almost
 $g$-compatible:
\begin{align}
\label{gco}X (g_V(Y,Z))=g_V(\widetilde\nabla^V_XY,Z)+g_V(Y,\widetilde\nabla^V_XZ)+2C_V(\widetilde\nabla^V_XV,Y,Z).
\end{align}
The equivalence follows easily from \eqref{nablaTV} and the fact that $\partial^\nu g=2C$.
This leads to the following anisotropic Koszul formula, by rotation through $ X,Y,Z $:
\begin{align}\nonumber
2 g_V(\widetilde\nabla^V_XY,Z)=\,\,&
X(g_V(Z,Y))+Y(g_V(X,Z))-Z(g_V(X,Y))\\&\nonumber
+g_V([Z,X],Y)+g_V(X,[Z,Y])+g_V(Z,[X,Y])\\&\label{Koszul0}
-2C_V(Z,Y,\widetilde\nabla^V_XV)-2C_V(X,Z,\widetilde\nabla^V_YV)+2C_V(X,Y,\widetilde\nabla^V_ZV).
\end{align}
\end{defi}
Once we have introduced the Levi-Civita-Chern connection of a pseudo-Finsler manifold, we can define the flag curvature in an easy way. The flag curvature turns out to be a very important geometric invariant of Finsler geometry. It is a measure of how geodesics get apart. This depends on the initial geodesic,  so its direction will be the flagpole $v\in A$. The plane where we consider  the variation of geodesics is determined by the flag $u\in T_{\pi(v)}M$.
\begin{defi}
The {\em flag curvature} $ \widetilde{K}_v $ of a pseudo-Finsler manifold $ (M,L) $ with flagpole $ v \in A $ and flag $ u \in T_{\pi(v)} M $  satisfying that $L(v) g_v(u,u)-g_v(v,u)^2 \neq 0 $ is given by 
\[
\widetilde{K}_v(u)
=
\frac{g_v(\widetilde{R}_v(v,u)u,v)}{L(v)g_v(u,u)-g^2_v(u,v)},
\]
where $\widetilde R$ is the curvature tensor of the Levi-Civita-Chern connection $\widetilde\nabla$ of $(M,L)$.
\end{defi}
 \begin{remark}
The flag curvature can be computed using other anisotropic connections as for example the Berwald one and, more generally, with any distinguished connection (see \cite[Proposition 3.6]{Jav19}). As we will see later, this not will be the case of the induced connection of a submanifold.
\end{remark} 
\section{Geometry of submanifolds in Finsler spaces}
\subsection{The second fundamental form and the induced connection}
Suppose that $(M,L)$ is a pseudo-Finsler manifold and $S\subset M$ a non-degenerate submanifold of $M$, namely, the restriction of $g_v$ into the tangent space $T_{\pi(v)}S$ is  non-degenerate for every $v\in TS\cap A$.  
 In such a case, $g_v$ allows us to obtain a decomposition of $T_{\pi(v)}M$ in the tangent and $g_v$-orthogonal part to $T_{\pi(v)}S$,
\[T_{\pi(v)}M=T_{\pi(v)}S\bigoplus (T_{\pi(v)}S)^{\bot_{g_v}}.\]
In the following, we will denote with the superindices $\top_{g_v}$  and $\bot_{g_v}$ the $g_v$-projection, respectively, into the tangent and $g_v$-orthogonal part to $T_{\pi(v)}S$. 
Now given $v\in A$ with $\pi(v)\in S$  and vector fields $X,Y\in \mathfrak{X}(S)$, it is possible to define $\widetilde{\nabla}^v_XY$  by considering  any choice of extensions $\widetilde X\in \mathfrak{X}(M)$ and $\widetilde Y\in \mathfrak{X}(M)$ of $X$ and $Y$, respectively. Then
\[
\widetilde{\nabla}^v_XY:=\widetilde{\nabla}^v_{\widetilde X}{\widetilde Y}
\]
is well-defined, namely, it does not depend on the choice of $\widetilde X$ and $\widetilde Y$. Using the above decomposition $T_{\pi(v)}M=T_{\pi(v)}S\bigoplus (T_{\pi(v)}S)^{\bot_{g_v}}$, one can express 
\begin{equation}\label{decomp}
\widetilde{\nabla}^v_XY=\widehat \nabla^v_XY+\IIs_v(X,Y),
\end{equation}
 where the tangent part to $S$, $\widehat \nabla^v_XY$, determines an $A\cap TS$-anisotropic connection on $S$ and $\IIs_v(X,Y)$ determines an anisotropic tensor in the following sense: for every $v\in A\cap TS$, it gives a map 
\[\IIs_v:T_{\pi(v)}S\times T_{\pi(v)}S \rightarrow (T_{\pi(v)}S)^{\bot_{g_v}}\] which is multilinear. Moreover, $\IIs_v$ is symmetric and $\widehat \nabla$ is torsion-free, since
\[\widetilde \nabla^v_XY-\widetilde\nabla^v_YX=[X,Y]\]
and as the Lie bracket $[X,Y]$ is tangent to $S$, it follows that $\widehat\nabla^v_XY-\widehat\nabla^v_YX=[X,Y]$ and $\IIs_v(X,Y)=\IIs_v(Y,X)$.

\subsection{The  Chern connection of $S$}
Let $(M,L)$ be a pseudo-Finsler manifold and $S\subset M$ a non-degenerate submanifold.  In the following, we will denote with $L|_S$ the restriction $L|_{TS\cap A}:TS\cap A\rightarrow \R$. It is easy to check that $L|_S$ is a pseudo-Finsler metric on $S$, because for every $v\in TS\cap A$  its fundamental tensor is the restriction of $g_v$ in \eqref{fundtensor} to $T_{\pi(v)}S\times T_{\pi(v)}S$, which is non-degenerate by hypothesis. 
Let $ \widetilde\nabla $ be the Levi-Civita-Chern connection of $(M,L)$ and $ \nabla $ the Levi-Civita-Chern connection of $(S,L|_S)$.
Let $Q$ the difference tensor between $\widehat\nabla$ introduced in \eqref{decomp} and $\nabla$, given for $v\in A\cap TS$ and  $X,Y\in \mathfrak{X}(S)$ by
\begin{align}\label{c2}
\widehat\nabla^v_X Y=\nabla^v_XY+Q_v(X,Y).
\end{align}
Observe that $Q$ determines an anisotropic tensor  ${\mathfrak T}^1_2(S,TS\cap A)$ with $Q: {\mathfrak T}^1_0(M,A)\times {\mathfrak T}^1_0(M,A)\rightarrow {\mathfrak T}^1_0(M,A)$ and one can define its
anisotropic tensor derivative $\nabla Q$ following \eqref{nablaTV2}.  In the next lemma, we will determine this tensor $Q$. The study of the relation between the induced and the intrinsic connections has been worked out 
to some extent in many of the references cited in the introduction for Gauss-Codazzi equations. Moreover, in \cite{CheYaYa00}, the authors study specifically the relation between the induced and the Chern connections without any mention of Gauss-Codazzi equations. 
\begin{lemma}\label{CartanQ} 
Let $(M,L)$ be a pseudo-Finsler manifold and $S$, a non-degenerate submanifold. For $v\in A\cap TS$ and $u,w,z\in T_{\pi(v)}S$,
\begin{enumerate}[(i)]
\item $Q_v(v,v)=0,$ 
\item $Q_v(u,v)=-( C^\flat_v(\II(v,v),u))^{\top_{g_v}}$, 
where   if $u_1,u_2,u_3\in T_{pi(v)}M$, $C^\flat_v$ is determined by  $C_v(u_1,u_2,u_3)=g_v(C^\flat_v(u_1,u_2),u_3)$,  namely, it is $g_v$-metrically equivalent to the Cartan tensor $C$.
\item the anisotropic tensor $Q$ is determined by
\begin{multline*}
g_v(Q_v(u,w),z)=- C_v(Q_v(u,v)+\II(u,v),w,z)
-C_v(Q_v(w,v)+\II(w,v),z,u)\\
+C_v(Q_v(z,v)+\II(z,v),u,w).
\end{multline*}
\item the anisotropic tensor $Q$ is symmetric.
\item  $(\partial^\nu Q)_v(v,v,u)=-2 Q_v(v,u)$. 
\end{enumerate}
\end{lemma}
\begin{proof}
 Part $(iii)$ is obtained directly by substracting the Koszul formulas for $\widetilde\nabla$ and $\nabla$.  
Parts $(i)$ and $(ii)$ are an immediate consequence of part $(iii)$ as by homogeneity the value of the Cartan tensor $C_v(u,w,z)$ is zero when one of the components coincides with $v$. Part $(iv)$ follows straightforwardly from part $(iii)$ and it can also be deduced analogously to the symmetry of $\IIs$. For part $(v)$, compute the derivative of $Q_{v+tu}(v+tu,v+tu)=0$ (which holds by part $(i)$) with respect to $t$ at $t=0$.
\end{proof}
\begin{lemma}\label{CartanQCor}
Let $(M,L)$ be a pseudo-Finsler manifold and $S$ a non-degenerate submanifold. For $v\in A\cap TS$ and $u,w,z\in T_{\pi(v)}S$,
\begin{enumerate}[(i)]
\item $g_v(Q_v(u,v),v) = g_v(Q_v(v,u),v)=0$.
\item $g_v(Q_v(u,w),v) =-g_v(Q_v(u,v),w) =-g_v(Q_v(v,u),w)=  C_v(u,w,\IIs_v(v,v))$.
\item The anisotropic  tensor $\partial^\nu Q$ is symmetric in the first two components and
\[g_v((\partial^\nu Q)_v(u,v,w),v) = g_v((\partial^\nu Q)_v(v,u,w),v) = 0.\]
\item  $(\nabla_w Q)_v$ is symmetric and  
\[g_v((\nabla_w Q)_v (u,v),v) =g_v((\nabla_w Q)_v (v,u),v)=0.\]
\end{enumerate}
\end{lemma}
\begin{proof}
Parts $(i)$ and $(ii)$ follow from part $(iii)$ of Lemma \ref{CartanQ}, taking into account the properties of the Cartan tensor. The symmetry of $\partial^\nu Q$ in the first components follows from the 
symmetry of $Q$ and the definition of vertical derivative.  To prove that $g_v((\partial^\nu Q)_v(u,v,w),v) =  0$, observe that
\begin{align*}
\frac{\partial}{\partial t}  g_{v+tw}(Q_{v+tw}(u,v+tw),v+tw)  \left. \vphantom{\frac{a}{a}} \right\vert_{t=0}
=& \,\, 
2 C_v(Q_v(u,v),v,w) 
\\&+
g_v((\partial^\nu Q)_v(u,v,w),v)
+
g_v(Q_v(u,w),v)
\\&+
g_v(Q_v(u,v),w)
\end{align*}
and the left hand side is zero by part $(i)$.   The first term to the right hand side is zero by the properties of the Cartan tensor and the third and fourth terms cancel by part $(ii)$,  which gives the nullity of the first term in part $(iii)$. The nullity of the second term, $g_v((\partial^\nu Q)_v(v,u,w),v) = 0$, follows now  from the nullity of the first one and the mentioned symmetry of $\partial^\nu Q$.   The symmetry of $(\nabla_wQ)_v$
follows from the symmetry of $Q$, the symmetry of $\partial^\nu Q$ in the first two components and \eqref{nablaTV2}. 
To prove that $g_v((\nabla_w Q)_v (u,v),v) =0$ in part $(iii)$, choose  extensions $ V , U $ of $ v , u $ such that $ \nabla^v V = 0 $ (recall \eqref{nablaXV=0}). Then  using \eqref{gco} and \eqref{nablaTV2}, we have
\begin{equation*}
w (g_V(Q_V(U,V),V))=
g_v((\nabla^v_w Q_v)(u,v),v)
+
g_v(Q_v(\nabla^v_wU,v),v),
\end{equation*}
where we have discarded all the terms in $\nabla^vV$, since it is zero by hypothesis.
Part $(i)$ implies that the left hand side and the second term of the right hand side are zero,
so the identity reduces to $g_v((\nabla^v_w Q_v)(u,v),v)=0$
as required. The other identity $g_v((\nabla^v_w Q_v)(v,u),v)=0$ follows from the last one using the symmetry of $(\nabla_wQ)_v$. 
\end{proof}
Let us define the derivative $\nabla\IIs$ of the second fundamental form $\IIs$ of $S$.  If $X,Y,Z\in {\mathfrak X}(S)$ and $v\in A\cap TS$, we define
\begin{equation} \label{c3}
(\nabla_X\IIs)_v(Y,Z):=(\widetilde\nabla^v_X(\IIs(Y,Z)))^{\bot_{g_v}}-\IIs_v(\nabla^v_XY,Z)-\IIs_v(Y,\nabla^v_XZ).
\end{equation}
This determines a map 	 $(\nabla_X\IIs)_v:{\mathfrak X}(S)\times {\mathfrak X}(S)\rightarrow (T_{\pi(v)}S)^{\bot_{g_v}}$ which is ${\mathcal F}(S)$-multi\-linear. 
Observe that $\widetilde\nabla^v_X(\IIs(Y,Z))$ is well-defined as it can be computed as a covariant derivative along an integral curve of $X$ (see \cite[\S 2.2]{Jav19}). Moreover, $v\mapsto \IIs_v(X,Y)$ is an anisotropic vector field. Indeed, it is well-defined for all $v\in A$ with $\pi(v)\in S$ and not only for those $v\in A\cap TS$. Its derivative is computed using an $A$-admissible extension $V$ of $v$ as 
\begin{equation}\label{secondc3} \widetilde\nabla^v_X(\IIs(Y,Z))=\widetilde\nabla^v_X(\IIs_V(Y,Z))-(\partial^\nu\IIs)_v(Y,Z,\widetilde\nabla^v_XV), 
\end{equation}
see \cite[Eqs. (8) and (9)]{Jav19}. 
Observe that the last term is well defined even though $\widetilde\nabla^v_XV$ is not necessarily tangent to $S$ as we have defined $\IIs_v$ for all $v$ in $\pi^{-1}(S)$.

\subsection{Totally geodesic submanifolds}  A very important class of submanifolds is made up of those whose geodesics are also geodesics of the background manifold. 
\begin{defi}
Let $(M,L)$ be a pseudo-Finsler manifold and $S$ a non-degenerate submanifold of $(M,L)$. We will say that $S$ is {\em totally geodesic} if the geodesics of $(S,L|_S)$ are also geodesics of $(M,L)$.
\end{defi}
Observe that in \cite{Li2010}, the author calls this familiy ``weakly totally geodesic'', saving the term ``totally geodesic'' for the family with second fundamental form $\IIs=0$. As the meaning of having a second fundamental form null everywhere is not clear to us, we prefer to use totally geodesic for the family introduced above. 
\begin{prop}\label{TotGeo}
A non-degenerate submanifold $S$ of a pseudo-Finsler manifold $(M,L)$ is totally geodesic if and only if one of the following two equivalent conditions holds:
\begin{enumerate}[(i)]
\item $\IIs_v(v,v)=0$ for all $v\in A\cap TS$.
\item $\IIs_v(v,u)=0$ for all $v\in A\cap TS$ and $u\in T_{\pi(v)}S$.
\end{enumerate}
In this case, $Q=0$.
\end{prop}
\begin{proof}
First of all observe that given $v_0\in A\cap TS$, we can choose a geodesic vector field $V$ of $(S,L|_S)$ in some open subset $\Omega$ of $S$ which extends $v_0$. Then by \eqref{decomp} and \eqref{c2}, and using that $V$ is a geodesic vector field of $(S,L|_S)$ and part $(i)$ of Lemma \ref{CartanQCor},
\[\widetilde\nabla^V_VV=\nabla^V_VV+Q_V(V,V)+\IIs_V(V,V)=\IIs_V(V,V),\]
which easily implies that the geodesic $\gamma_{v_0}$ of $(S,L|_S)$ with initial velocity $v_0$ is also a geodesic of $(M,L)$ if and only if $\IIs_v(v,v)=0$ for all $v$ tangent to the geodesic $\gamma_{v_0}$. This implies straightforwardly that $S$ is totally geodesic if and only if part $(i)$ holds. Let us prove that $(i)$ implies $(ii)$ (the converse is trivial). Observe that deriving $\IIs_{v+tu}(v+tu,v+tu)=0$ with respect to $t$, it follows that
\[(\partial^\nu\IIs)_v(v,v,u)=-2\IIs_v(v,u).\]
To conclude, we only have to prove that $(\partial^\nu\IIs)_v(v,v,u)=0$.
  Again from \eqref{decomp} and \eqref{c2}, we have
\[\widetilde\nabla^v_XY=\nabla^v_XY+Q_v(X,Y)+\IIs_v(X,Y),\]
for vector fields $X,Y\in {\mathfrak X}(S)$ and $v\in A\cap TS$. Computing the vertical derivative of the above identity follows that
\[\widetilde P_v(u,w,z)=P_v(u,w,z)+(\partial^\nu Q)_v(u,w,z)+(\partial^\nu\IIs)_v(u,w,z)\]
for $v\in TS\cap A$ and $u,w,z\in T_{\pi(v)}S$, and observing that $\widetilde P_v(v,v,u)=P_v(v,v,u)=0$ (see for example \cite[Lemma 3.5]{Jav19}), and $(\partial^\nu Q)_v(v,v,u)=0$ by parts $(ii)$ and $(v)$ of Lemma \ref{CartanQ}, we conclude that $(\partial^\nu\IIs)_v(v,v,u)=0$ as required. The last statement about $Q$ follows from part $(ii)$ and $(iii)$ of Lemma \ref{CartanQ}.
\end{proof}
Observe that the equivalence between parts  $(i)$ and $(ii)$ in the last proposition can be also found in the proof of Theorem 4.4 in \cite{Li2010}.
\subsection{The Gauss and Codazzi Equations}
Considering a non-degenerate submanifold $S$ of a pseudo-Finsler manifold $(M,L)$ as in the previous section, recall that $\nabla$ is the Levi-Civita-Chern connection of $(S,L|_S)$, $\widetilde\nabla$ is the Levi-Civita-Chern connection of $(M,L)$ and $\widehat \nabla$ is the induced connection by $\widetilde\nabla$ on $S$ (see \eqref{decomp}). We will denote by $R, \widetilde{R}$ and $\widehat{R}$ the curvature tensors of $\nabla$, $\widetilde\nabla$ and $\widehat\nabla$, repectively
(see \eqref{Rv}). Moreover,  $P$ and $\widetilde P$ are the vertical derivatives of $\nabla$ and $\widetilde \nabla$, respectively (see \eqref{tensorP}), $Q$ is the difference tensor of $\widehat\nabla$ and $\nabla$ (see \eqref{c2}), $\IIs$ the second fundamental form of $S$ (see \eqref{decomp}) and $\nabla \IIs$ is the derivative of the second fundamental form as defined in \eqref{c3}.
\begin{theorem}
\label{J}
With the above notation, if $(M,L)$ is a pseudo-Finsler manifold and $S\subset M$ a non-degenerate submanifold, for $ v \in A\cap TS $, $u,w,z,b\in T_{\pi(v)}S$ and $n\in (T_{\pi(v)}S)^{\bot_{g_v}}$,
\begin{align}\label{GEQ}
\,g_{v}\left(\widetilde{R}_v(u,w)z,b\right)&=
g_v(R_v(u,w)z,b)\nonumber \\&\quad-g_{v}\left(\IIs_v(w,z),\IIs_v(u,b)\right)+g_{v}\left(\IIs_v(u,z),\IIs_v(w,b)\right)\nonumber\\
&\quad -2C_v(\IIs_v(u,v),\IIs_v(w,z),b)+2C_v(\IIs_v(w,v),\IIs_v(u,z),b)\nonumber\\
&\quad+g_{v}\left(\widetilde{P}_v(u,z,\IIs_v(w,v))-\widetilde{P}_v(w,z,\IIs_v(u,v)),b\right)\nonumber\\
&\quad+g_{v}\left(P_v(u,z,Q_v(w,v))-P_v(w,z,Q_v(u,v)),b\right)\nonumber\\
&\quad +g_{v}\left((\nabla_uQ)_v(w,z)-(\nabla_wQ)_v(u,z),b\right)
\nonumber\\&\quad+g_{v}\left((\partial^\nu Q)_v(u,z,Q_v(w,v))-(\partial^\nu Q)_v(w,z,Q_v(u,v)),b\right)\nonumber\\&
\quad+g_{v}\left(Q_v(u,Q_v(w,z))-Q_v(w,Q_v(u,z)),b\right),\\
\label{Codazzi} \,g_{v}\left(\widetilde{R}_v(u,w)z,n\right)&=g_{v}\left((\nabla_u\IIs)_v(w,z),n\right)-g_{v}\left((\nabla_w\IIs)_v(u,z),n\right)\nonumber\\&\quad
+g_v(\IIs_v(u,Q_v(w,z)),n)-g_v(\IIs_v(w,Q_v(u,z)),n)
\nonumber\\&\quad+g_v((\partial^\nu\IIs)_v(w,z,\IIs_v(u,v)),n)-g_v((\partial^\nu\IIs)_v(u,z,\IIs_v(w,v)),n)\nonumber
\\&\quad+g_{v}\left(\widetilde{P}_v(u,z,\IIs_v(w,v)),n\right)-g_{v}\left(\widetilde{P}_v(w,z,\IIs_v(u,v)),n\right).
\end{align}  
\end{theorem}
\begin{proof}
Let $X,Y$ and $Z$ be vector fields tangent to $S$ which are extensions of $u,w,z$, respectively, and such that $[X,Y]=[X,Z]=[Y,Z]=0$, and let $V$ be an $A$-admissible extension of $v$ (in a neighborhood of $\pi(v)$) such that $\widehat \nabla^vV=0$ (recall \eqref{nablaXV=0}).  It follows that $\widetilde\nabla^v V=\IIs_v(\cdot,V)$.
Taking into account the last identity, let us compute $\widetilde{R}_v$ using the affine connection $\widetilde\nabla^V$ (see \eqref{RvRV}):
\begin{align}\label{p0}
\widetilde{R}_v(X,Y)Z=&\widetilde{R}^V(X,Y)Z|_{\pi(v)}-\widetilde{P}_v(Y,Z,\IIs_V(X,V))+\widetilde{P}_v(X,Z,\IIs_V(Y,V))\nonumber\\
=&\widetilde\nabla^v_X\nvm_YZ-\widetilde\nabla^v_Y\nvm_XZ-\widetilde{P}_v(Y,Z,\IIs_V(X,V))+\widetilde{P}_v(X,Z,\IIs_V(Y,V)).
\end{align}
Then,  by \eqref{decomp}, we have 
\[\widetilde\nabla^v_X\nvm_YZ=
\widetilde\nabla^v_X\nvi_YZ+\widetilde\nabla^v_X(\IIs_V(Y,Z))=\widehat\nabla^v_X\nvi_YZ+\IIs_v(X,\widehat\nabla^V_YZ)+\widetilde\nabla^v_X(\IIs_V(Y,Z))\] and analogously $\widetilde\nabla^v_Y\nvm_XZ=\widehat\nabla^v_Y\nvi_XZ+\II(Y,\widehat\nabla^V_XZ)+\widetilde\nabla^v_Y(\IIs_V(X,Z))$.
Combining these identities with 
\[\widehat{R}^V(X,Y)Z|_{\pi(v)}=\widehat\nabla^v_X\nvi_YZ-\widehat\nabla^v_Y\nvi_XZ=
\widehat{R}_v(X,Y)Z,\]
(for the last identity recall that $\widehat\nabla^v V=0$) gives us:
\begin{multline}\label{p2}
\widetilde{R}^V(X,Y)Z|_{\pi(v)}=\widehat{R}_v(X,Y)Z
+\IIs_v(X,\widehat\nabla^V_YZ)+\widetilde\nabla^v_X(\IIs_V(Y,Z))
\\-\II(Y,\widehat\nabla^V_XZ)-\widetilde\nabla^v_Y(\IIs_V(X,Z)).
\end{multline}
Let $B$ be an extension of $b$ tangent to $S$.
By  \eqref{gco} and as $g_{V}(\IIs_V(Y,Z),B)=0$ along $S$, we have
\begin{align}\label{p4}
g_{v}\left(\nvm_X (\IIs_V(Y,Z)),B\right)&=-g_{v}(\IIs_V(Y,Z),\nvm_XB)-2C_v(\widetilde\nabla^V_XV,\IIs_V(Y,Z),B)\nonumber\\&=-g_{v}\left(\IIs_V(Y,Z),\IIs_V(X,B)\right)-2C_v(\IIs_V(X,V),\IIs_V(Y,Z),B).
\end{align}
Similarly,
\begin{align}\label{p5}
g_{v}\left(\nvm_Y (\IIs_V(X,Z)),B\right)=-g_{v}\left(\IIs_V(X,Z),\IIs_V(Y,B)\right)-2C_v(\IIs_V(Y,V),\IIs_V(X,Z),B).
\end{align}
Putting together \eqref{p0}, \eqref{p2}, \eqref{p4} and \eqref{p5} and taking into account the expression of $\widehat R$ in terms of $R$ and $Q$ in \cite[Prop. 2.16]{Jav19} gives us \eqref{GEQ}.
In order to obtain \eqref{Codazzi}, 
by \eqref{c3} and \eqref{secondc3},
\begin{align} \nonumber
g_v(\widetilde\nabla^V_X(\IIs_V(Y,Z)),n)
= \,\, &
g_v((\nabla^v_X(\IIs(Y,Z)))^{\bot_{g_v}},n)
+
g_v((\partial^\nu\IIs)_v(Y,Z,\widetilde\nabla^v_XV),n)
\\ \nonumber = \,\, &
g_v((\nabla_X\IIs)_v(Y,Z),n)
+
g_v(\IIs_v(\nabla^v_XY,Z),n)
\\ &
\qquad+
g_v(\IIs_v(Y,\nabla^v_XZ),n)	
+
g_v((\partial^\nu\IIs)_v(Y,Z,\widetilde\nabla^v_XV),n)
\label{p6}
\end{align}
and analogously
\begin{align} \nonumber
g_v(\widetilde\nabla^V_Y(\IIs_V(X,Z)),n)
= \,\, &
g_v((\nabla_Y\IIs)_v(X,Z),n)
+
g_v(\IIs_v(\nabla^v_YX,Z),n)
\\ &
\qquad+
g_v(\IIs_v(X,\nabla^v_YZ),n)
+
g_v((\partial^\nu\IIs)_v(X,Z,\widetilde\nabla^v_YV),n)
\label{p7}
\end{align}
From \eqref{p0}, \eqref{p2}, \eqref{p6}, \eqref{p7} and recalling that the Lie brackets of $X,Y,Z$ are assumed to be zero,$\nabla$ is torsion-free and $\widetilde\nabla^v_\cdot V=\IIs_v(\cdot,V)$, we conclude \eqref{Codazzi}.
\end{proof}
\begin{corollary}\label{Curflag}
Let $(M,L)$ be a pseudo-Finsler manifold and $S$ a non-degenerate submanifold. Then the flag curvature of $(S,L|_S)$ with flagpole  $v\in A\cap TS$ and flag $u\in T_{\pi(v)}S$ with $L(v)g_v(u,u)-g_v(v,u)^2\not=0$ is given by
\begin{align}
K_v(u)=&\widetilde K_v(u)+\frac{g_v(\IIs_v(u,u),\IIs_v(v,v))-g_v(\IIs_v(v,u),\IIs_v(v,u))}{L(v)g_v(u,u)-g_v(v,u)^2}\nonumber \\
&+\frac{ g_v(\widetilde P_v(u,u,\IIs_v(v,v))-(\nabla_vQ)_v(u,u),v)+C_v(u,Q_v(v,u),\IIs_v(v,v))}{L(v)g_v(u,u)-g_v(v,u)^2}.\label{flagcurvature}
\end{align}
\end{corollary}
\begin{proof}
By \eqref{GEQ}, we have 
\begin{align}
\,g_{v}\left(\widetilde{R}_v(v,u)u,v\right)&=
g_v(R_v(v,u)u,v)\nonumber -g_{v}\left(\IIs_v(u,u),\IIs_v(v,v)\right)+g_{v}\left(\IIs_v(v,u),\IIs_v(u,v)\right)\nonumber\\
&\quad -2C_v(\IIs_v(v,v),\IIs_v(u,u),v)+2C_v(\IIs_v(u,v),\IIs_v(v,u),v)\nonumber\\
&\quad+g_{v}\left(\widetilde{P}_v(v,u,\IIs_v(u,v))-\widetilde{P}_v(u,u,\IIs_v(v,v)),v\right)\nonumber\\
&\quad+g_{v}\left(P_v(v,u,Q_v(u,v))-P_v(u,u,Q_v(v,v)),v\right)\nonumber\\
&\quad +g_{v}\left((\nabla_uQ)_v(u,u)-(\nabla_wQ)_v(v,u),v\right)
\nonumber\\&\quad+g_{v}\left((\partial^\nu Q)_v(v,u,Q_v(u,v))-(\partial^\nu Q)_v(u,u,Q_v(v,v)),v\right)\nonumber\\&
\quad+g_{v}\left(Q_v(v,Q_v(u,u))-Q_v(u,Q_v(v,u)),v\right).\label{firstaprox}
\end{align}
In the last identity many terms are zero. The terms in $C_v$ are zero because of the properties of Cartan tensor, the first term in $\widetilde P_v$ and the first term in $P_v$ are both  zero by \cite[Eq. (56)]{Jav19} since $\widetilde \nabla$ and $\nabla$ are the Levi-Civita-Chern connections of $(M,L)$ and $(S,L|_S)$, respectively. All the terms in $Q$ are zero by Lemma \ref{CartanQCor} except two, namely, $g_{v}\left((\nabla_uQ)_v(u,u),v\right)$ and $-g_v\left(Q_v(u,Q_v(v,u)),v\right)$. The last term coincides with $-C_v(u,Q_v(v,u),\IIs_v(v,v))$ by part $(ii)$ of Lemma \ref{CartanQCor}. Putting all this together, \eqref{firstaprox} becomes
\begin{align}
\,g_{v}\left(\widetilde{R}_v(v,u)u,v\right)&=
g_v(R_v(v,u)u,v)\nonumber-g_{v}\left(\IIs_v(u,u),\IIs_v(v,v)\right)+g_{v}\left(\IIs_v(v,u),\IIs_v(u,v)\right)\nonumber\\
&\quad+g_{v}\left(-\widetilde{P}_v(u,u,\IIs_v(v,v))+(\nabla_uQ)_v(u,u),v\right)-C_v(u,Q_v(v,u),v)\nonumber,
\end{align}
which by the definition of flag curvature leads to \eqref{flagcurvature}.
\end{proof}
With this expression of the flag curvature, we can reobtain \cite[Theorem 4.4]{Li2010}.
\begin{corollary}\label{CurTotGeo}
When $S$ is totally geodesic, its flag curvature coincides with the one of the background manifold.
\end{corollary}
\begin{proof}
Being $S$ totally geodesic, by Proposition \ref{TotGeo} we know that $\IIs_v(v,u)=0$ for all $v\in TS\cap A$ and $u\in T_{\pi(v)}S$ and $Q=0$, so the statement follows from  \eqref{flagcurvature}.
\end{proof}
\section{Randers-Minkowski spaces}
Let $\V$ be a vector space of dimension $n$ and recall that a {\em Minkowski norm} on $\V$ is a non-negative function $F:\V\rightarrow[0,+\infty)$ smooth away from $0$, which is positive homogeneous of degree 1 and such that, for every $v\in\V\setminus\{0\}$, its fundamental tensor $g_v$, defined
as in \eqref{fundtensor} for $L=F^2$, is positive definite.  We will say that the pair $(\V,F)$ is a {\em Minkowski space}. Let $h:\V\times\V\rightarrow \R$ be a positive definite scalar product on $\V$. If $W\in \V$ is a vector such that $h(W,W)<1$, then we can translate the indicatrix of $h$ with $W$. It turns out that 
the translated hypersurface as it is strongly convex and it contains the origin of $\V$ is the indicatrix of a Minkowski norm $F$ (see \cite[Theorem 2.14]{JavSan14}). The Minkowski norm $F$ is determined by the property
\[h(v/F(v)-W,v/F(v)-W)=1,\]
which is equivalent to a second degree polynomial equation in $F(v)$ and  it implies that 
\begin{equation}\label{zermelo}
F(v)
=
\frac{1}{\lW}\left(- h( v , W)
+
\sqrt{h( v , W)^2+ \ell(W)h(v , v)}\right),
\end{equation}
where $\lW=1-h(W,W)$. This family of Minkowski norms has been used to solve the Zermelo problem of navigation when the wind does not depend on time\footnote{When the wind is time-depending, the Zermelo problem can be studied using Finsler spacetimes (see \cite[\S 6.3.1]{JavSan20}).} \cite{BRS04,Sh03}. Moreover, it coincides with the family of Randers norms, which are constructed as follows. Given a positive definite scalar product $g$ on $\V$ and a one-form $\beta$ with $g$-norm less than one, then the function $R:\V\rightarrow \R$ given by
\begin{equation}\label{alphabet}
R(v)=\alpha(v)+\beta(v)
\end{equation}
for $v\in\V$, where $\alpha(v)=\sqrt{g(v,v)}$,  is a Minkowski norm which is said to be of {\em Randers type}. It is clear that a Zermelo norm as in \eqref{zermelo} is of Randers type. The converse is also true (see \cite[\S 1.3]{BRS04}). Indeed, if $\beta(v)=g(v,B)$, then the Zermelo data of $R$ is given by $h(u,w)=\ell(B)(g(u,w)-g(u,B)g(w,B))$, where $\ell(B)=1-g(B,B)$, and $W=-B/\ell(B)$ (see also \cite[Prop. 3.1]{BiJav11}).
In the following, we will say $(h,W)$ is the Zermelo data of the Randers norm $R$. Moreover, if $F$ is a Minkowski norm on $\V$ of Zermelo-Randers type, we will say the pair $(\V,F)$ is a {\em Randers-Minkowski space}.

From now on, given a Randers norm $F$, we will denote by $g$ and $C$ the fundamental and the Cartan tensors of $L=F^2$, respectively, and  by
$\Sigma=\{v\in \V: F(v)=1\}$, the indicatrix of $F$.

\begin{lemma} \label{gVProptoSigmaN}
Let $F$ be a Randers norm on $\V$ with Zermelo data $ (h, W) $. Given $v\in \V\setminus\{0\}$, define
\begin{equation}\label{mu}
\phi(v)=h(v/F(v)-W , v/F(v)).
\end{equation}
Then
\begin{enumerate}[(i)]
\item $ g_v(v/F(v),\cdot) = \frac{1}{\phi(v)} h( v/F(v)-W,\cdot) $.  Moreover, $T_{v/F(v)}\Sigma$ is the $g_v$-orthogonal space to $v$ and the $h$-orthogonal space to $v/F(v)-W$. 
\item $ g_v = \frac{1}{\phi(v)} h $ on $ T_{v/F(v)} \Sigma \times T_{v/F(v)} \Sigma $,
\item 	$g_v(v/F(v),W)=\frac{\phi(v)-1}{\phi(v)}$.
\end{enumerate}
Furthermore, if   $u\in T_{v/F(v)}\Sigma$ (or equivalently $g_v(v,u)=0$),  then
	\begin{equation}\label{usefuleq}
	h(v, u)= \phi(v)^2 F(v) g_v(u,W).
	\end{equation}
\end{lemma}
\begin{proof}
To check $(i)$, observe that $g_v(v/F(v),\cdot)$ and $\h(v/F(v)-W,\cdot)$ are two one-forms on $\V$ with the same kernel. Namely, the kernel of $g_v(v/F(v),\cdot)$ is the tangent space to the indicatrix of $F$  as $d_vF^2=2g_v(v,\cdot)$,  and the indicatrix of $F$ is the unit sphere of $h$ translated with $W$ at $v/F(v)$, while the kernel of $h(v/F(v)-W,\cdot ) $ is the tangent space at  $v/F(v)-W$  to the unit sphere of $h$. We know then that  $g_v(\frac{v}{F(v)},\cdot)=\mu(v)h(v/F(v)-W,\cdot) $ for a certain function $\mu:\V\rightarrow \R$. It follows that 
$g_v(v/F(v),v)=\mu(v)h(v/F(v)-W,v )$ and as $g_v(v/F(v),v)=F(v)$, we conclude that $\mu(v)=\frac{1}{\phi(v)}$, where $\phi$ is defined in \eqref{mu}.  The last statement of part $(i)$ is now straightforward. 
%

For $(ii)$, recall that $g_v|_{T_{v/F(v)}\Sigma\times T_{v/F(v)}\Sigma}$ is the second fundamental form of $\Sigma$ at $\tfrac{v}{F(v)}$ with respect to $-\tfrac{v}{F(v)}$ (see for example \cite[Eq. (2.5)]{JavSan14}) and $h|_{T_{v/F(v)}\Sigma\times T_{v/F(v)}\Sigma}$ is the second fundamental form of $\Sigma$ at $v/F(v)$ with respect to $-v/F(v)+W$, which is an $h$-unit vector  (recall that, by part $(i)$, $-v/F(v)+W$ is $h$-orthogonal to $T_{v/F(v)}\Sigma$).  If $\widetilde \nabla$ is the natural affine connection of $\V$ and $u,w\in T_{v/F(v)}\Sigma$, then the second fundamental form $\sigma^\xi$ with respect to the transverse vector $\xi$  is defined as follows. Let $X,Y$ be two vector fields which are extensions of $u,w$, respectively. Then the identity
\[(\widetilde\nabla_XY)_{v/F(v)}= \sigma(u,w) \xi+P_\sigma\]
with $P_\sigma$ tangent to $T_{v/F(v)}\Sigma$
determines $\sigma$ by the uniqueness of the decomposition. In the case of  $g_v|_{T_{v/F(v)}\Sigma\times T_{v/F(v)}\Sigma}$ and $h|_{T_{v/F(v)}\Sigma\times T_{v/F(v)}\Sigma}$, this implies that 
\begin{equation}\label{seconfundforms}
g_v(u,w)\left(-\frac{v}{F(v)}\right)+P_g=h(u,w)\left(-\frac{v}{F(v)}+W\right)+P_h
\end{equation}
with $P_g$ and $P_h$ tangent to $T_{v/F(v)}\Sigma$. Moreover, as $-v/F(v)+W$ is $h$-orthogonal to $T_{v/F(v)}\Sigma$ by part $(i)$,  $-v/F(v)=\phi(v)(-v/F(v)+W)+P$, with $P$ tangent to $T_{v/F(v)}\Sigma$, which implies that
\begin{align*}
g_v(u,w)\left(-\tfrac{v}{F(v)}\right)+P_g=&g_v(u,w)\left(\phi(v)(-v/F(v)+W)+P\right)+P_g\\
=&\phi(v)g_v(u,w)\left(-\tfrac{v}{F(v)}+W\right)+g_v(u,w)P+P_g.
\end{align*}
By the uniqueness of the decomposition and \eqref{seconfundforms},  the identity $h(u,w)=\phi(v)g_v(u,w)$ follows, which is equivalent to part $(ii)$.

%

For $(iii)$, observe that, by part $(i)$, $g_v(v/F(v),W)=\frac{1}{\phi(v)}h(v/F(v)-W,W).$
As $h(v/F(v)-W,v/F(v)-W)=1$, the conclusion follows.

Finally, for \eqref{usefuleq},	recall that, by part $(i)$, being $g_v$-orthogonal to $v$ is equivalent to being $h$-orthogonal to $v/F(v)-W$, namely, the vector space $T_{v/F(v)}\Sigma$ coincides with the $h$-orthogonal vectors to  $v/F(v)-W$. As 
	 $h(v/F(v)-W,v/F(v)-W)=1$, using part $(ii)$, it follows that 
\begin{align*}
h(v,u)
=&
h( v-h(v/F(v)-W,v)(v/F(v)-W),u)
\\=&
\phi(v)g_v(v-F(v)\phi(v) (v/F(v)-W),u) 
\\=&
-\phi(v)^2F(v) g_v(v/F(v)-W,u)=\phi(v)^2F(v) g_v(W,u).
\end{align*}
	\end{proof}
	\begin{prop} \label{CartangSigma}
Let $F:\V\rightarrow\R$ be a Randers norm with Zermelo data $(h,W)$.
Then for $ v \in \V\setminus\{0\}$ and $ u,w,z \in T_{v/F(v)}\Sigma $,
\begin{align*}
C_v(u,w,z)
= & \,\,
-
\frac{1}{2\phi(v)^2F(v)^2}(h(u,w) h( z,v ) + h(w,z) h( u,v ) +h(z,u) h( w,v ) ).
\end{align*}
\end{prop}
\begin{proof}
Recall that $C_v(u,w,z)=\frac{1}{2} \frac{\partial}{\partial t} g_{v+tz}(u,w)|_{t=0}$. To apply Lemma \ref{gVProptoSigmaN}, let us define 
\begin{equation}\label{help-1}
\varphi_t(u)=u-g_{v+tz}(u,(v+tz)/F(v+tz))\frac{v+tz}{F(v+tz)},
\end{equation}
and observe that $\varphi_t(u)\in T_{(v+tz)/F(v+tz)}\Sigma$ for all $t\in\R$. 
Moreover, using that $d_vF^2=2 g_v(v,\cdot)$,
\begin{equation}\label{help1}
\frac{\partial}{\partial t}F(v+tz)|_{t=0}=\frac{1}{2F(v)} dF^2_v(z)=\frac{1}{F(v)}g_v(v,z)=0,
\end{equation}
and using the above identity,
\begin{align}\nonumber
\frac{\partial}{\partial t}g_{v+tz}(u,(v+tz)/F(v+tz))|_{t=0}&=2 C_v(u,v/F(v),z)+g_v(u,z/F(v))\\&= \frac{1}{F(v)}g_v(u,z).\label{help3}
\end{align} 
On the other hand, using  again \eqref{help1}, 
\begin{equation}\label{help2}
\frac{\partial}{\partial t}\phi(v+tz)|_{t=0}=h(\tfrac{z}{F(v)},\tfrac{v}{F(v)})+h(\tfrac{v}{F(v)}-W,\tfrac{z}{F(v)})=\tfrac{1}{F(v)^2}h(z,v),
\end{equation}
because $z$ is $h$-orthogonal to $\frac{v}{F(v)}-W$ (recall part $(i)$ of Lemma \ref{gVProptoSigmaN}). 
Finally, from \eqref{help1} and \eqref{help3} and part $(ii)$ of Lemma \ref{gVProptoSigmaN},
\begin{equation}\label{help4}
\frac{\partial}{\partial t}\varphi_t(u)|_{t=0}=-\frac{1}{F(v)^2}g_v(u,z)v =-\frac{1}{\phi(v)F(v)^2}h(u,z)v.
\end{equation} 
 Analogously,
\begin{equation}\label{help5}
\frac{\partial}{\partial t}\varphi_t(w)|_{t=0}=-\frac{1}{\phi(v)F(v)^2}h(w,z)v.
\end{equation} 
Let us compute the Cartan tensor, taking into account \eqref{help-1},
\begin{align*}\nonumber
C_v(u,w,z)=&\frac{1}{2} \frac{\partial}{\partial t} g_{v+tz}(u,w)|_{t=0}\\
=&\frac{1}{2} \frac{\partial}{\partial t} g_{v+tz}(\varphi_t(u),\varphi_t(w))|_{t=0}\\
&+\frac{1}{2} \frac{\partial}{\partial t} g_{v+tz}(u,(v+tz)/F(v+tz))g_{v+tz}(w,(v+tz)/F(v+tz)))|_{t=0}.
\end{align*}
Observe that the second term is the derivative of a product of two functions which are zero in $t=0$. Then its value is zero in $t=0$. As a consequence, and using part $(ii)$ of Lemma \eqref{gVProptoSigmaN}, and then
\eqref{help2}, \eqref{help4} and \eqref{help5},
\begin{align*}\nonumber
C_v(u,w,z)=&\frac{1}{2} \frac{\partial}{\partial t} g_{v+tz}(\varphi_t(u),\varphi_t(w))|_{t=0}=\frac{1}{2} \frac{\partial}{\partial t} \frac{1}{\phi(v+tz)}h(\varphi_t(u),\varphi_t(w))|_{t=0}\\
=&-\frac{1}{2\phi(v)^2F(v)^2}h(z,v)h(u,w)-\frac{1}{2\phi(v)^2F(v)^2}h(u,z)h(v,w)\\&-\frac{1}{2\phi(v)^2F(v)^2}h(w,z)h(v,u),
\end{align*}
as desired.
\end{proof}
\subsection{Submanifolds of a Randers-Minkowski space}	
In this section, we will consider a submanifold $S$ of a Randers-Minkowski space $(\V,F)$ with Zermelo data $(h,W)$. Our main goal is to express all the Randers geometric invariants of $S$ in terms of the invariants with respect to $h$. Let us begin by obtaining the Zermelo data of the induced metric on $S$.  We will denote by $F|_S$ the restriction $F|_{TS}:TS\rightarrow \R$,  which is a Finsler metric on $S$, namely, $\br F|_S\er^2$ is a  pseudo-Finsler metric defined in the whole tangent bundle and with positive definite fundamental tensor. It is well-known that the Levi-Civita-Chern connection of $(\V,F)$ coincides with the Levi-Civita connection of $(\V,h)$. Indeed, this is true for any Minkowski norm $F$, not necessarily of Randers type.  In the following, we will use the superindices $\top_p$ and $\bot_p$ to denote the $h$-projection to $T_{p}S$ and its $h$-orthogonal space, respectively.  Moreover,   $W^\top$ will denote a tangent vector field to $S$ such that $(W^\top)_p=W^{\top_p}$ and $W^\bot$ will denote an $h$-orthogonal vector field along $S$ with an analogous convention. 
\begin{prop}\label{ZermSub}
Let $(\V, F)$ be a Randers-Minkowski space with Zermelo data $(h,W)$, and $S\subset \V$, a submanifold. Then the Zermelo data of $(S,F|_S)$ is given by $(\frac{1}{1-h(W^{\bot},W^{\bot})}h,W^{\top})$.
\end{prop}
\begin{proof}
Given $v\in TS\setminus \bf 0$, we know that $F(v)$ is determined by $h(\tfrac{v}{F(v)}-W,\tfrac{v}{F(v)}-W)=1$. As $W$ is not necessarily tangent to $S$, the last identity does not allow us to obtain the Zermelo data of $(S,F|_S)$. Recall that  $W=W^{\top_{\pi(v)}}+W^{\bot_{\pi(v)}}$ is the decomposition in the tangent and $h$-orthogonal part  to $T_{\pi(v)}S$. Then
\[ h(\tfrac{v}{F(v)}-W^{\top_{\pi(v)}},\tfrac{v}{F(v)}-W^{\top_{\pi(v)}})+h(W^{\bot_{\pi(v)}},W^{\bot_{\pi(v)}})=1,\]
which implies that
\[\frac{1}{1-h(W^{\bot_{\pi(v)}},W^{\bot_{\pi(v)}})}h(\tfrac{v}{F(v)}-W^{\top_{\pi(v)}},\tfrac{v}{F(v)}-W^{\top_{\pi(v)}})=1,\]
namely, the indicatrix of the induced metric $F|_S$ is the displacement of the indicatrix of $\frac{1}{1-h(W^{\bot_{\pi(v)}},W^{\bot_{\pi(v)}})}h$ with $W^{\top_{\pi(v)}}$. This is equivalent to having as Zermelo data  $(\frac{1}{1-h(W^\bot,W^\bot)}h,W^\top)$.
\end{proof}
In the following, given $v\in\V\setminus\{0\}$, we will denote by $\WSigP $ the $h$-projection of $W$ into $T_{\pi(v)}S\cap T_{v/F(v)}\Sigma$. 
\begin{lemma} \label{IIvII}
Let $(\V, F)$ be a Randers-Minkowski space with Zermelo data $(h,W)$, and $S\subset \V$, a submanifold. If $\IIs'$ is the second fundamental form with respect to $h$, its second fundamental form $ \IIs $ with respect to $F$  (defined in \eqref{decomp}) is given by
\begin{equation}\label{SSFrel}
\IIs_v(u,w)=\IIs'(u,w)+\frac{1}{\phi(v)}h( \IIs'(u,w),W) \left( \smash{\tfrac{v}{F(v)}}-\WSigP \right),
\end{equation}
for $ v \in \V\setminus\{0\} $ and $ u,w \in T_{\pi(v)} S $. 
\end{lemma}
\begin{proof}
Let $ X,Y \in \mathfrak X(S) $ be extensions of $ u,w\in T_{\pi(v)}S$, respectively, and observe that $ \IIs_v(u,w)
= (\tilde\nabla^v_XY)^{\bot_{g_v}}
= (\IIs'(u,w))^{\bot_{g_v}}
= \IIs'(u,w) - \IIs'(u,w)^{\top_{g_v}} $. Let $\lbrace \smash{\tfrac{v}{F(v)}},e_2,e_2,\ldots,e_{r} \rbrace$ be a $ g_v $-orthonormal basis of $ T_{\pi(v)} S $. In such a case $ e_2,\ldots,e_{r} \in T_{\tfrac{v}{F(v)}} \Sigma $. Then
\begin{equation}\label{step0}
\IIs'(u,w)^{\top_{g_v}}
= g_v(\IIs'(u,w),\smash{\tfrac{v}{F(v)}}) \smash{\tfrac{v}{F(v)}}
+ \sum_{i=2}^{r} g_v(\IIs'(u,w),e_i) e_i.
\end{equation}
From part $(i)$ of  Lemma \ref{gVProptoSigmaN} and using that $\IIs'(u,w)$ is $h$-orthogonal to $ T_{\pi(v)} S$,
\begin{equation}\label{intereq}
g_v(\IIs'(u,w),\tfrac{v}{F(v)})
= \tfrac{1}{\phi(v)} h( \IIs'(u,w),\tfrac{v}{F(v)}-W)
= -\tfrac{1}{\phi(v)} h( \IIs'(u,w),W).
\end{equation}
Now observing that $\IIs'(u,w)+h(\IIs'(u,w),W) (v/F(v)-W)$ is tangent to $T_{v/F(v)}\Sigma$ because it is $h$-orthogonal to $v/F(v)-W$ (recall part $(i)$ of Lemma \ref{gVProptoSigmaN}), and using part $(ii)$ of Lemma \ref{gVProptoSigmaN}, it follows that
\begin{align*}
g_v(\IIs'(u,w),e_i)=&g_v(\IIs'(u,w)+h(\IIs'(u,w),W) (v/F(v)-W),e_i)\\&-h(\IIs'(u,w),W)g_v(v/F(v)-W,e_i)\\
=& \frac{1}{\phi(v)}h(\IIs'(u,w)+h(\IIs'(u,w),W) (v/F(v)-W),e_i)\\
&+h(\IIs'(u,w),W)g_v(W,e_i).
\end{align*}
Moreover, $\IIs'(u,w)$ is $h$-orthogonal to $T_{\pi(v)}S$, by definition and $v/F(v)-W$ is $h$-orthogonal to $T_{v/F(v)}\Sigma$. Therefore, both vectors are $h$-orthogonal to $e_i$, and then using also \eqref{usefuleq} and the fact that $h(v/F(v),e_i)=h(W,e_i)$ (because $v/F(v)-W$ is $h$-orthogonal to $e_i$),
\begin{align}\nonumber
g_v(\IIs'(u,w),e_i)=h(\IIs'(u,w),W)g_v(W,e_i)&=\frac{1}{F(v)\phi(v)^2}h(\IIs'(u,w),W)h(v,e_i)\\
&=\frac{1}{\phi(v)^2}h(\IIs'(u,w),W)h(W,e_i).\label{gII}
\end{align}
Using \eqref{intereq} and \eqref{gII} in \eqref{step0}, and taking into account that $e_2,\ldots,e_r$ is an $h$-orthogonal basis of $T_{v/F(v)}\Sigma\cap T_{\pi(v)}S$, with $h(e_i,e_i)=\phi(v)$ (recall part $(ii)$ of Lemma \ref{gVProptoSigmaN}), we obtain
\begin{align*}
\IIs'(u,w)^{\top_{g_v}}&=-\tfrac{1}{\phi(v)} h( \IIs'(u,w),W)\smash{\tfrac{v}{F(v)}}+\sum_{i=2}^r\frac{1}{\phi(v)^2}h( \IIs'(u,w),W)h(W,e_i)e_i\\
&=-\tfrac{1}{\phi(v)} h( \IIs'(u,w),W)(\smash{\tfrac{v}{F(v)}}-\WSigP).
\end{align*}
As we have seen above that $\IIs_v(u,w)=\IIs'(u,w)-\IIs'(u,w)^{\top_{g_v}}$, \eqref{SSFrel} follows.
\end{proof}

\begin{lemma} \label{tedious01}
Let $(\V, F)$ be a Randers-Minkowski space with Zermelo data $(h,W)$, and $S\subset \V$, a  submanifold. Given $v\in\V\setminus\{0\}$ and $ u \in T_{\pi(v)} S $ such that $ g_v(u,v) = 0 $, if $\IIs_v$ is the second fundamental form of $(S,F|_S)$ with respect to $v$ and $Q$ is the difference tensor between the induced and Levi-Civita-Chern connection of $S$, then
\begin{align} 
\label{tedious01eq2}
C_v(u,u,\IIs_v(v,v))
= & \,\,
-
\frac{1}{2F(v)^2\phi(v)}g_v(u,u) h( \IIs_v(v,v),v),\\
\label{tedious01eq1}
C_v(u,Q_v(u,v),\IIs_v(v,v))
= & \,\,
-
\frac{1}{4F(v)^4\phi(v)^2}g_v(u,u) h( \IIs_v(v,v),v )^2.
\end{align}
\end{lemma}
\begin{proof}
Observe that by definition of $\IIs_v$, $g_v(v,\IIs_v(v,v))=g_v(u,\IIs_v(v,v))=0$. Then by part $(ii)$ of Lemma \ref{gVProptoSigmaN}, $h(u,\IIs_v(v,v))=0$ and as a consequence, from Proposition \ref{CartangSigma},
\[C_v(u,u,\IIs_v(v,v))
=  \,\,
-
\frac{1}{2F(v)^2\phi(v)^2}h(u,u) h( \IIs_v(v,v),v).\]
Using again  part $(ii)$ of Lemma \ref{gVProptoSigmaN} to obtain that $h(u,u)=\phi(v) g_v(u,u)$, we conclude \eqref{tedious01eq2}.

Noting that $ g_v(Q_v(u,v),v) = 0 $ by part $(i)$ of Lemma \ref{CartanQCor}, we can apply again Proposition  \ref{CartangSigma}. Moreover, $ g_v(Q_v(u,v),\IIs_v(v,v)) =0$ and  by part $(ii)$ of Lemma \ref{gVProptoSigmaN}, $h(Q_v(u,v),\IIs_v(v,v)) =0$. As we have seen above that $h(u,\IIs_v(v,v))=0$, from Proposition \ref{CartangSigma},
\begin{equation}\label{pass1}
C_v(u,Q_v(u,v),\IIs_v(v,v))= -\frac{1}{2F(v)^2\phi(v)^2} h(u,Q_v(u,v)) h(\IIs_v(v,v),v).
\end{equation}
Finally, applying part $(ii)$ of Lemma \ref{gVProptoSigmaN} and part $(ii)$ of Lemma \ref{CartanQCor}, 
\[h(u,Q_v(u,v))=\phi(v) g_v(u,Q_v(u,v))=-\phi(v)C_v(u,u,\IIs_v(v,v)).\]
Taking into account \eqref{tedious01eq2} and then substituting in \eqref{pass1}, we obtain \eqref{tedious01eq1}.
\end{proof}
\subsection{Flag curvature}
We are ready to compute the flag curvature of a submanifold using \eqref{flagcurvature}. We will need the derivative of the connection $\IIs'$. 
\begin{defi} \label{euclideanIIs}
 Let $(\V,F)$ be a Randers-Minkowski space with Zermelo data $(h,W)$ and $S$, a submanifold of $\V$. Denote by $ \IIs' $ the \emph{second fundamental form} with respect to $h$ of $S $, and by $\bar\nabla$, the  induced connection on $S$ (computed using $h$). 
Given $X,Y,Z\in\mathfrak{X}(S)$, let us define $\bar\nabla \IIs' $ as follows
\begin{equation} \label{c4}
(\bar\nabla_X\IIs')(Y,Z) = (\widetilde\nabla_X(\IIs'(Y,Z)))^\bot-\IIs'(\bar\nabla_XY,Z)-\IIs'(Y,\bar\nabla_XZ).
\end{equation}
\end{defi}
\begin{lemma} \label{tedious02}
Let $(\V,F)$ be a Randers-Minkowski space  with Zermelo data $ (h,W)$ and  $S$, a submanifold of $\V$. For $ v \in \V\setminus\{0\} $ and $ u \in T_{\pi(v)} S $ such that $g_v(v,u)=0$,
\begin{align} \nonumber
g_v((\nabla_v Q)_v(u,u),v) = \,\, & \nonumber
-\frac{g_v(u,u)}{2F(v)\phi(v)^2}\left| \tfrac{v}{F(v)}-\WSigP\right|_h^2 \left[ \vphantom{\frac{v}{F(v)}}h( (\bar\nabla_v \IIs')(v,v),W )\right.\\&
\hspace{-0.5cm}-
h( \IIs'(v,v),\IIs'(v,W^{\top_{\pi(v)}}))	-\frac{4}{\phi(v)} h( \IIs'(v,v),W) \h( \IIs'(v,\WSigP),W)\nonumber \\ 
&\quad \quad\left.+ \tfrac{1}{\phi(v)^2F(v)}h( \IIs'(v,v),W)^2 (4\phi(v)-2\left| \tfrac{v}{F(v)}-\WSigP\right|_h^2)\right].
\label{nablaQ}
\end{align}
\end{lemma}
\begin{proof}
Recall that it is possible to choose a local extension $V$ of $v$ 
 such that $ \nabla^v V = 0$ (see \eqref{nablaXV=0}). Then we can also choose a local extension $U$ of $u$  such that $ [U,V] = 0$. In such a case, we also have that $\nabla^v_VU=\nabla^v_UV+[U,V]|_{\pi(v)}=0$. 
Using that $\nabla^v_vV=\nabla^v_vU=0$, Lemma \ref{CartanQCor} and \eqref{tedious01eq2}, it follows that
\begin{align*} \nonumber
g_v((\nabla_v Q)_v(u,u),v) &
= v(g_V(Q_V(U,U),V))=\,\, 
v (C_V(U,U,\IIs_V(V,V)) )\\
&=- v \left(\frac{1}{2(F\circ V)^2\phi\circ V}g_V(U,U) h( \IIs_V(V,V),V )\right).
\end{align*}
 As $ v((F\circ V)^2) = 2g_v(\nabla^v_v V,v) = 0 $ and $ v(g_V(U,U)) = 2g_v(\nabla^v_v U,u) = 0 $, and using Lemma  \ref{IIvII}, we have
\begin{align}\nonumber
g_v((\nabla_v Q)_v(u,u),v)=  & - v \left(\frac{g_V(U,U)}{2(F\circ V)^2(\phi\circ V)^2}h( \IIs'(V,V),W)h( V/(F\circ V)-\WSigPV,V)\right) \\
= & -\frac{g_v(u,u)}{2F(v)^2\phi(v)^2} \left(-2\frac{v(\phi\circ V)}{\phi(v)} h(\IIs'(v,v),W )h( \tfrac{v}{F(v)}-\WSigP,v)\right.\nonumber \\
&+v(h(\IIs'(V,V),W))h(\tfrac{v}{F(v)}-\WSigP,v)\nonumber \\
&\left.+h( \IIs'(v,v),W ) v(h(V/(F\circ V)-\WSigPV,V))
 \right).
\label{tedious02eq0}
\end{align}
Using that $ \nabla^v V = 0 $ and part $(i)$ of Lemma \ref{CartanQ},
\begin{equation}\label{nablatildeV}
 \widetilde\nabla_v V = \widehat\nabla^v_v V + \IIs_v(v,v) = Q_v(v,v)+\nabla^v_v V + \IIs_v(v,v) = \IIs_v(v,v) 
 \end{equation}
  is $ g_v $-orthogonal to $ T_{\pi(v)} S \ni v $, therefore $h( \frac{v}{F(v)}-W,\widetilde\nabla_v V ) = 0 $. Using also Lemma \ref{IIvII},  and recalling \eqref{mu} and 
  that $\widetilde\nabla$ is also the Levi-Civita connection of $h$, 
  we have
\begin{align}\nonumber 
v (\phi\circ V)
&=
v(h( \tfrac{V}{F\circ V}-W,\tfrac{V}{F\circ V}))=
\tfrac{1}{F(v)}
h( \widetilde\nabla_v V,\tfrac{v}{F(v)})
=
\tfrac{1}{F(v)} h(\IIs_v(v,v),\tfrac{v}{F(v)})\\ \label{vlambda}
&= \tfrac{1}{F(v)\phi(v)}h( \IIs'(v,v),W) h(\tfrac{v}{F(v)} -\WSigP,\tfrac{v}{F(v)}).
\end{align}
By Definition \ref{euclideanIIs}, 
\begin{align*}
v (h( \IIs'(V,V),W))
= \,\, &
v (h(\IIs'(V,V),W^{\bot}))
\\= \,\, &
h( \widetilde\nabla_v( \IIs'(V,V)),W^{\bot_{\pi(v)}})
+
h( \IIs'(v,v),\widetilde\nabla_v (W^{\bot}) )
\\ = \,\, &
h( (\bar\nabla_v \IIs')(v,v),W )
+
2 h( \IIs'(\bar\nabla_v V,v),W )
+
h( \IIs'(v,v),\widetilde\nabla_v (W^{\bot})).
\end{align*}
As $ W $ is constant, $ \widetilde\nabla_v (W^{\bot}) =  -\widetilde\nabla_v (W^{\top}) $ and 
\[h( \IIs'(v,v),\widetilde\nabla_v (W^{\bot})) = -h(\IIs'(v,v),\IIs'(v,W^{\top_{\pi(v)}})).\]
Furthermore,  $ \bar\nabla_v V = (\widetilde\nabla_v V)^{\top_{\pi(v)}} = (\IIs_v(v,v))^{\top_{\pi(v)}} = \frac{1}{\phi(v)} h(\IIs'(v,v),W) (\frac{v}{F(v)}-\WSigP) $ by \eqref{nablatildeV} and Lemma \ref{IIvII}, therefore
\[
h( \IIs'(\bar\nabla_vV,v),W)
=
\frac{1}{\phi(v)}h( \IIs'(v,v),W) h(\IIs'(\tfrac{v}{F(v)}-\WSigP,v),W),
\]
which leads to
\begin{align} \label{tedious02eq8}
 v( h( \IIs'(V,V),W)) =  & 
h( (\bar\nabla_v \IIs')(v,v),W)
-
h( \IIs'(v,v),\IIs'(v,W^{\top_{\pi(v)}}) )\\  \nonumber &
+
\frac{2}{\phi(v)}h( \IIs'(v,v),W) h( \IIs'(\tfrac{v}{F(v)}-\WSigP,v),W).
\end{align}
Let $ \tfrac{V}{F\circ V},E_2,\cdots,E_r $ be a $ g_V $-orthonormal local frame of $S$ in a neighborhood of $ \pi(v) $ such that $ \nabla^v_v E_i = 0 $. This can be obtained, for example, by making the $\nabla^V$-parallel translation of a $g_v$-orthonormal basis of $T_{\pi(v)}S$ along the integral curves of $V$, since we can assume that the integral curve of $V$ passing through $\pi(v)$ is a geodesic of $(S,L|_S)$. Therefore $ \widetilde\nabla_v E_i = \IIs_v(v,E_i)+Q_v(v,E_i) $ for $i=2,\ldots,r$.  Then  $E_2,\ldots, E_r$ at every $p\in S$ where it is defined is an $h$-orthogonal frame of $T_{\pi(V_p)} S\cap T_{V_p/F(V_p)}\Sigma$ and $h(E_i,E_i) =\phi\circ V$ (recall part $(ii)$ of Lemma \ref{gVProptoSigmaN}). This implies that 
\begin{equation}\label{wsigp}
\WSigPV=\frac{1}{\phi\circ V}\sum_{i=2}^r h(W,E_i) E_i.
\end{equation}
 As a consequence,
\begin{align}
\label{tedious02eq7}
v (h(\tfrac{V}{F\circ V}-\WSigPV,V))=v\left(\frac{1}{F\circ V}h( V,V) -\frac{1}{\phi\circ V}\sum_{i=2}^r h( W,E_i) h( E_i,V)\right).
\end{align}
In the following, we will use that $T_{V_p/F(V_p)}\Sigma$ coincides with $\{V_p/F(V_p)-W\}^{\bot_{V_p}}$ at every $p\in S$ where $V$ is defined, and then, in particular, $V/(F\circ V)-W$ is $h$-orthogonal to $E_i$, which implies that $\h(W,E_i)=\frac{1}{F\circ V}h( V,E_i)$  and then 
\[h( W,E_i) h( E_i,V)=F\circ V h(W,E_i)^2.\]  Taking this into account and  \eqref{nablatildeV}, we get
\begin{align}\nonumber
v (h( \tfrac{V}{F\circ V}-\WSigPV,V))=&\frac{2}{F(v)}h( \IIs_v(v,v),v)+\frac{v(\phi\circ V)}{\phi(v)^2}\sum_{i=2}^r h(W,E_i) h(E_i,v)\\
&  -\frac{2F(v)}{\phi(v)}\sum_{i=2}^r h(W,\IIs_v(v,E_i)+Q_v(v,E_i)) h(E_i,W). \label{tedious02eq71}
\end{align}
Moreover, using  \eqref{vlambda}  and \eqref{wsigp}, 
\begin{align}
\frac{v(\phi\circ V)}{\phi(v)^2}\sum_{i=2}^r h(W,E_i) h(E_i,v)
 &= \frac{1}{\phi(v)^2F(v)^2}h(\IIs_v(v,v),v) h( \sum_{i=2}^r h( W,E_i) E_i,v)\nonumber \\
&=  \frac{1}{\phi(v)F(v)^2}h( \IIs_v(v,v),v) h( \WSigP,v). \label{tedpart2}
\end{align}
On the other hand, using again \eqref{wsigp} and that $v/F(v)-W$ is $h$-orthogonal to $T_{v/F(v)}\Sigma$, and observing that $g_v(\IIs_v(v,\WSigP)+Q_v(v,\WSigP),v)=0$ (because of the definition of $\IIs_v$ and part $(i)$ of Lemma \ref{CartanQCor}) and then $\IIs_v(v,\WSigP)+Q_v(v,\WSigP)\in T_{v/F(v)}\Sigma$,
\begin{align}
\nonumber -\frac{2F(v)}{\phi(v)}\sum_{i=2}^r h( W,\IIs_v(v,E_i)+&Q_v(v,E_i)) h( E_i,W)\\
=& -2F(v) h( W,\IIs_v(v,\WSigP)+Q_v(v,\WSigP) ) \nonumber \\
=&  -2 h( v,\IIs_v(v,\WSigP)+Q_v(v,\WSigP) ). \label{tedpart3}
\end{align}
Observe that using \eqref{usefuleq}, that $E_i$ is $h$-orthogonal to $V/(F\circ V)-W$, and \eqref{wsigp},
\begin{align}
\nonumber W^{\top_{g_v}}=&g_v(\frac{v}{F(v)},W) \frac{v}{F(v)}+\sum_{i=2}^r g_v(W,E_i)E_i\\
\nonumber &=g_v(\frac{v}{F(v)},W) \frac{v}{F(v)}+\sum_{i=2}^r \frac{1}{F(v)\phi(v)^2}h(v,E_i) E_i\\
\nonumber &=g_v(\frac{v}{F(v)},W) \frac{v}{F(v)}+\sum_{i=2}^r \frac{1}{\phi(v)^2}h(W,E_i) E_i\\
&=g_v(\frac{v}{F(v)},W) \frac{v}{F(v)}+\frac{1}{\phi(v)}\WSigP. \label{WagWsig}
\end{align}
Now applying  \eqref{usefuleq}, \eqref{WagWsig}, parts $(i)$ and $(ii)$ of Lemma \ref{CartanQCor}, and Proposition \ref{CartangSigma},
\begin{align*}
h( Q_v(v,&\WSigP),v) = \phi(v)^2F(v) g_v(Q_v(v,\WSigP),W)\nonumber \\
&=  \phi(v)^2F(v) g_v(Q_v(v,\WSigP),W^{\top_{g_v}})=F(v)\phi(v) g_v(Q_v(v,\WSigP),\WSigP)\nonumber\\\nonumber
&=- F(v) \phi(v) C_v(\WSigP,\WSigP,\IIs_v(v,v))\\&=\frac{1}{2F(v)\phi(v)}h(\WSigP,\WSigP) h(\IIs_v(v,v),v).
\end{align*}
Using the above identity in  \eqref{tedpart3} and then the resulting identity and \eqref{tedpart2} in \eqref{tedious02eq71}, it follows that
\begin{align*}\nonumber
v (h(\tfrac{V}{F\circ V}-\WSigPV,V))=&\frac{2}{F(v)}h( \IIs_v(v,v),v)+ \frac{1}{\phi(v)F(v)^2}h( \IIs_v(v,v),v) h( \WSigP,v)\\
&  -2 h( v,\IIs_v(v,\WSigP) )-\frac{1}{F(v)\phi(v)}h(\WSigP,\WSigP) h(\IIs_v(v,v),v).
\end{align*}
Observing that
$h(\tfrac{v}{F(v)},\WSigP)=h( W,\WSigP)=h( \WSigP,\WSigP)$ and applying Lemma \ref{IIvII}, we have
\begin{align}\nonumber
v (h(\tfrac{V}{F\circ V}-\WSigPV,V))&=\frac{2}{F(v)}h( \IIs_v(v,v),v) -2 h( v,\IIs_v(v,\WSigP) )\\
&\hspace{-1.2cm}=\frac{2}{\phi(v)}\left(\frac{1}{F(v)}h( \IIs'(v,v),W) - h( W,\IIs'(v,\WSigP) )\right)h( \frac{v}{F(v)}-\WSigP,v).
\label{tedious02eq71b}
\end{align}
Using \eqref{vlambda}, \eqref{tedious02eq8} and \eqref{tedious02eq71b} in \eqref{tedious02eq0}, we obtain 
\begin{align*} \nonumber
g_v((\nabla_v Q)_v(u,u),v) 
=& -\frac{g_v(u,u)}{2F(v)^2\phi(v)^2}( -\frac{2}{\phi(v)^2} h(\IIs'(v,v),W )^2h( \tfrac{v}{F(v)}-\WSigP,\tfrac{v}{F(v)})^2\\
&+h( \frac{v}{F(v)}-\WSigP,v)\left[\vphantom{\frac{a}{F(a)}}h( (\bar\nabla_v \IIs')(v,v),W)\right.\\  \nonumber &
-
h( \IIs'(v,v),\IIs'(v,W^{\top_{\pi(v)}}) )\\ \nonumber
&+
\frac{2}{\phi(v)}h( \IIs'(v,v),W) h( \IIs'(\tfrac{v}{F(v)}-\WSigP,v),W)\\
&\hspace{-0.7cm}\left.+\frac{2}{\phi(v)}h( \IIs'(v,v),W)\left(\frac{1}{F(v)}h( \IIs'(v,v),W) - h( W,\IIs'(v,\WSigP) )\right)\right]).
\end{align*}
Finally, a straightforward simplification of the above identity leads to \eqref{nablaQ} taking into account that $h( \tfrac{v}{F(v)}-\WSigP,\WSigP)=0$, as commented above, and then 
\begin{equation}\label{hvW}
h( \tfrac{v}{F(v)}-\WSigP,\tfrac{v}{F(v)})
=h( \tfrac{v}{F(v)}-\WSigP,\tfrac{v}{F(v)}-\WSigP).
\end{equation}
\end{proof}
 We have already all the information to express the flag curvature of a submanifold $S$ in terms of elements related to the Zermelo data $(h,W)$. More precisely, 
apart from $h$ and $W$, we will use the second fundamental form $\IIs'$ of $S$ computed with $h$, and its derivative $\bar\nabla \IIs'$, where $\bar \nabla$ is the $h$-induced connection of $S$, and the $h$-projections of
$W$, $W^\top$ and $\WSigP$ to $T_pS$ and $T_pS\cap T_{v/F(v)}\Sigma$, respectively. 
\begin{theorem}\label{maintheorem}
Let $(\V, F)$ be a Randers-Minkowski space with Zermelo data $ (h,W) $ and $ S$, a  submanifold of $\V$. For $ v \in TS\setminus\bf 0 $ and $ u \in T_{\pi(v)} S $,
\begin{align}
\nonumber K_v(u)=&(h(\tfrac{v}{F(v)},\tfrac{v}{F(v)})-h(\tfrac{\tilde u}{|\tilde u|_h},\tfrac{v}{F(v)})^2)K^h(v,u)\\
&+\frac{\left| \tfrac{v}{F(v)}-\WSigP\right|_h^2}{\phi(v)^2F(v)^2h(\tilde u,\tilde u)} (h( \IIs'(u,u),W)h(\IIs'(v,v),W )- h(\IIs'(u,v),W )^2)\nonumber\\
\nonumber&\hspace{-0.4cm}+\frac{1}{2F(v)^3\phi(v)^2}\left| \tfrac{v}{F(v)}-\WSigP\right|_h^2 \left[h( (\bar\nabla_v  \IIs')(v,v),W )
-
h( \IIs'(v,v),\IIs'(v,W^{\top_{\pi(v)}}) )\right.\\ \nonumber
&\quad \quad+ \tfrac{1}{\phi(v)^2F(v)}h( \IIs'(v,v),W)^2(4\phi(v)-\frac{5}{2}\left| \tfrac{v}{F(v)}-\WSigP\right|_h^2)\\
&\quad \quad\hspace{3cm}\left.-\frac{4}{\phi(v)} h( \IIs'(v,v),W) h( \IIs'(v,\WSigP),W))\right],\label{Kflagformula}
\end{align}
where $K^h(v,u)$ is the Riemannian sectional curvature in the plane ${\rm span}\{v,u\}$ computed with the metric induced by $h$ on $S$,  $\phi(v)=h(v/F(v)-W,v/F(v))$  and $\tilde u=u-\frac{1}{\phi(v)}h(v/F(v)-W,u)v/F(v)$.
\end{theorem}
\begin{proof}
Assume that $g_v(v,u)=0$ and 
recall that, by Corollary \ref{Curflag}, taking into account that $\tilde K_v(u)=0$ and $\tilde P=0$ in a Randers-Minkowski space,
\begin{align}
\nonumber K_v(u)=&\frac{1}{\phi(v)}\frac{h( \IIs_v(u,u),\IIs_v(v,v) )- h(\IIs_v(u,v),\IIs_v(u,v) )}{g_v(u,u)F(v)^2}\\
& +\frac{1}{g_v(u,u)F(v)^2}\left(
C_v(u,Q_v(u,v),\IIs_v(v,v))
-
g_v((\nabla_v Q)_v(u,u),v))\right).\label{Kflag1}
\end{align}
Using Lemma \ref{IIvII}, it follows that 
\begin{align}
\nonumber h( \IIs_v(u,u)&,\IIs_v(v,v) )- h(\IIs_v(u,v),\IIs_v(u,v) )=h( \IIs'(u,u),\IIs'(v,v) )- h(\IIs'(u,v),\IIs'(u,v) )\\
&+\frac{1}{\phi(v)^2}\left| \tfrac{v}{F(v)}-\WSigP\right|_h^2 (h( \IIs'(u,u),W)h(\IIs'(v,v),W )- h(\IIs'(u,v),W )^2).\label{Kflag2}
\end{align}
Moreover, as we have assumed that  $g_v(v,u)=0$, we can apply part $(ii)$ of Lemma \ref{gVProptoSigmaN} to compute
\[\phi(v)(g_v(u,u)F(v)^2-g_v(u,v)^2)=F(v)^2h(u,u),\]
and then
\begin{equation}\label{Kflag3}
\frac{h(u,u)h(v,v)-h(u,v)^2}{\phi(v)(g_v(u,u)F(v)^2-g_v(u,v)^2)}=h(\tfrac{v}{F(v)},\tfrac{v}{F(v)})-h(\tfrac{u}{|u|_h},\tfrac{v}{F(v)})^2.
\end{equation}
Given an arbitrary $u\in T_{\pi(v)}S$, then $\tilde u=u-g_v(v/F(v),u)v/F(v)$ is a vector such that $\{\tilde u,v\}$ generates the same plane as $\{u,v\}$ and $g_v(\tilde u,v)=0$. Using part $(i)$ of Lemma \ref{gVProptoSigmaN}, one has
\[\tilde u=u-g_v(v/F(v),u)v/F(v)=u-\frac{1}{\phi(v)}h(v/F(v)-W,u)v/F(v).\]
Taking into account \eqref{Kflag2} and \eqref{Kflag3}, it follows that the first term in \eqref{Kflag1} becomes the two first terms to the right hand side in \eqref{Kflagformula}.
From \eqref{tedious01eq1},  \eqref{hvW} and Lemma \ref{IIvII},
\begin{align*}
C_v(u,Q_v(u,v),\IIs_v(v,v))=&-
\frac{1}{4\phi(v)^4F(v)^4}g_v(u,u) h( \IIs'(v,v),W )^2h(\tfrac{v}{F(v)}-\WSigP,v)^2\\
=&-\frac{g_v(u,u)}{4\phi(v)^4F(v)^2} h( \IIs'(v,v),W )^2\left|\tfrac{v}{F(v)}-\WSigP\right|_h^4,
\end{align*}
and then using Lemma 	 \ref{tedious02}, we conclude that the last term  in \eqref{Kflag1} coincides with the three last terms in \eqref{Kflagformula}.
\end{proof}
\begin{corollary}
Let $(\V, F)$ be a Randers-Minkowski space with Zermelo data $ (h,W) $ and $ S$, a  submanifold of $\V$. Then $S$ is of scalar flag curvature if and only if for every $v\in A$, 
\begin{multline*}
(h(\tfrac{v}{F(v)},\tfrac{v}{F(v)})-h(\tfrac{u}{|u|_h},\tfrac{v}{F(v)})^2)K^h(v,u)\\
+\frac{\left| \tfrac{v}{F(v)}-\WSigP\right|_h^2}{\phi(v)^2F(v)^2h(u,u)} (h( \IIs'(u,u),W)h(\IIs'(v,v),W )- h(\IIs'(u,v),W )^2)
\end{multline*}
has the same value for every $u\in T_{\pi(v)}S$ such that $g_v(v,u)=0$.
\end{corollary} 
\subsection{Hypersurfaces in Randers-Minkowski spaces} We will consider now the case in that $S$ is a hypersurface of $\V$. In such a case, given $v\in TS\setminus\bf 0$, there exists a unique vector $\xi_v\in \V$ (up to a sign) such that $\xi_v$ is $g_v$-orthogonal to $T_{\pi(v)}S$ and $g_v(\xi_v,\xi_v)=1$. Moreover, there exists $\sigma_v:T_{\pi(v)}S\times T_{\pi(v)}S\rightarrow \R$ such that
\[\IIs_v(u,w)=\sigma_v(u,w)\xi_v\]
for all $u,w\in T_{\pi(v)}S$. Analogously, for every $p\in S$, there exists $N_p\in\V$ such that $N_p$ is $h$-orthogonal to $T_pS$ and $h(N_p,N_p)=1$, and $\sigma'_p:T_{p}S\times T_{p}S\rightarrow \R$ such that
\[\IIs'(u,w)=\sigma'_p(u,w)N_p\]
for all $u,w\in T_{p}S$.
\begin{lemma}\label{hyperSSF}
Let $S$ be a hypersurface of a Randers-Minkowski space $(\V,F)$. With the above notation,
\begin{equation}\label{xiv}
\xi_v=\frac{\sqrt{\phi(v)}}{\sqrt{1-h(N_{\pi(v)},W)^2}}\left(N_{\pi(v)}+h(N_{\pi(v)},W)\left(\tfrac{v}{F(v)}-W\right)\right)
\end{equation}
is the $g_v$-orthogonal vector to $T_{\pi(v)}S$ with $g_v(\xi_v,\xi_v)=1$, and
\begin{equation}\label{sigmav}
\sigma_v(u,w)=\frac{\sigma'(u,w)}{\sqrt{\phi(v) (1-h(N_{\pi(v)},W)^2)}}
\end{equation}
for all $v\in TS\setminus\bf 0$ and $u,w\in T_{\pi(v)}S$. 
\end{lemma}
\begin{proof}
Let us show that $\xi_v$ is $g_v$-orthogonal to $T_{\pi(v)}S$. First observe that $h(N_{\pi(v)},v)=0$ by definition, and then $h(N_{\pi(v)},W)=-h(N_{\pi(v)},\tfrac{v}{F(v)}-W)$. This implies that $\xi_v$ is $h$-orthogonal to $\tfrac{v}{F(v)}-W$  (recall that $h(\tfrac{v}{F(v)}-W,\tfrac{v}{F(v)}-W)=1$),  and then by part $(i)$ of Lemma \ref{gVProptoSigmaN}, $g_v$-orthogonal to $v$. If $u\in T_{v/F(v)}\Sigma\cap T_{\pi(v)}S$, then by part $(ii)$ of Lemma \ref{gVProptoSigmaN},  $g_v(\xi_v,u)=\frac{1}{\phi(v)}h(\xi_v,u)=0$, because $u$ is $h$-orthogonal to $N_{\pi(v)}$ and$\tfrac{v}{F(v)}-W$, since it lies respectively in $T_{\pi(v)}S$ and $T_{v/F(v)}\Sigma$. This concludes that $\xi_v$ is $g_v$-orthogonal to $T_{\pi(v)}S$, because  $T_{v/F(v)}\Sigma\cap T_{\pi(v)}S$ has dimension $\dim S-1$ and $v\notin T_{v/F(v)}\Sigma$. Observing that $g_v(\xi_v,\xi_v)=\frac{1}{\phi(v)}h(\xi_v,\xi_v)$ by part $(ii)$ of Lemma \ref{gVProptoSigmaN}, it is straightforward to check that $g_v(\xi_v,\xi_v)=1$. Finally, as $\IIs'(u,w)-\IIs_v(u,w)$ is tangent to $T_{\pi(v)}S$ by Lemma \ref{IIvII}, it follows that $\sigma'(u,w)N_{\pi(v)}=\IIs'(u,w)=\IIs_v(u,w)^{\bot_{\pi(v)}}=\sigma_v(u,w) \xi_v^{\bot_{\pi(v)}}=\sigma_v(u,w) h(N_{\pi(v)},\xi_v)N_{\pi(v)},$ and then one has $\sigma'(u,w)=\sigma_v(u,w)h(N_{\pi(v)},\xi_v)$. As $h(N_{\pi(v)},\xi_v)=\sqrt{ \phi(v) (1-h(N_{\pi(v)},W)^2)}$, the last identity is equivalent to \eqref{sigmav}, which concludes.
\end{proof}
\begin{corollary}
Let $(\V, F)$ be a Randers-Minkowski space with Zermelo data $ (h,W) $ and $ S$, a hypersurface of $\V$. For $ v \in TS\setminus\bf 0 $ and $ u \in T_{\pi(v)} S $,
\begin{align}
\nonumber (1-h(N_{\pi(v)},W)^2)K_v(u)=&(h(\tfrac{v}{F(v)},\tfrac{v}{F(v)})-h(\tfrac{\tilde u}{|\tilde u|_h},\tfrac{v}{F(v)})^2)K^h(v,u)\\
\nonumber&+\frac{1}{2F(v)^3}\left[\bar\nabla_v \sigma'(v,v)h(N_{\pi(v)},W )
-\sigma'(v,v)\sigma'(v,W^{\top_{\pi(v)}}) \right.\\ \nonumber
&+ \tfrac{1}{\phi(v)F(v)}h(N_{\pi(v)},W)^2\sigma'(v,v)^2(4-\frac{5\phi(v)}{2(1-h(N_{\pi(v)},W )^2)})\\
&\quad \quad\left.-\frac{4}{\phi(v)} h(N_{\pi(v)} ,W)^{2} \sigma'(v,v)\sigma'(v,\WSigP)\right],\label{Kflagformula2}
\end{align}
where $K^h(v,u)$ is the Riemannian curvature in the plane $\{v,u\}$ computed with the metric induced by $h$ on $S$, and $\tilde u=u-\frac{1}{\phi(v)}h(v/F(v)-W,u)v/F(v)$.
\end{corollary}
\begin{proof}
Let us compute the flag curvature using \eqref{flagcurvature}. The second term to the right hand side can be computed using Lemma \ref{hyperSSF}, resulting
\begin{equation}
\frac{\sigma'(u,u)\sigma'(v,v)-\sigma'(v,u)^2}{\phi(v)(1-h(N_{\pi(v)},W)^2)L(v)g_v(u,u)}
\end{equation}
choosing $u$ such that $g_v(v,u)=0$. Proceeding as in \eqref{Kflag3}, we obtain the first term to the right hand side in \eqref{Kflagformula2}. Now observe that 
\begin{equation}\label{Wsigp}
\WSigP=W-\frac{\phi(v)h(N_{\pi(v)},W)}{1-h(N_{\pi(v)},W)^2}N_{\pi(v)}-\left(-1+\frac{\phi(v)}{1-h(N_{\pi(v)},W)^2}\right)\left(\tfrac{v}{F(v)}-W\right).
\end{equation}
This follows from the fact that $T_{\pi(v)}S\cap T_{v/F(v)}\Sigma=\{N_{\pi(v)},\tfrac{v}{F(v)}-W\}^{\bot_h}$ and that the right hand side is $h$-orthogonal to both, $N_{\pi(v)}$ and $\tfrac{v}{F(v)}-W$ as it can be easily checked. From \eqref{Wsigp}, one gets immediately,
\[\tfrac{v}{F(v)}-\WSigP=-\frac{\phi(v)}{1-h(N_{\pi(v)},W)^2}\left(h(N_{\pi(v)},W)N_{\pi(v)}+\tfrac{v}{F(v)}-W\right),\]
and then
\[  |\tfrac{v}{F(v)}-\WSigP|_h^2=\frac{\phi(v)^2}{1-h(N_{\pi(v)},W)^2}.\]
Taking into account the above identity and that $\bar\nabla_X\IIs'(Y,Z)=(\bar\nabla_X\sigma')(Y,Z)N$ for any $X,Y,Z\in \mathfrak{X}(S)$, we obtain the other terms of \eqref{Kflagformula2} from the last terms of \eqref{Kflagformula}.
\end{proof}
\begin{corollary}\label{hyperclassi}
A hypersurface $S$ of a Randers-Minkowski space $(\V, F)$ with Zermelo data $ (h,W) $ is of scalar flag curvature if and only if
\[(h(\tfrac{v}{F(v)},\tfrac{v}{F(v)})-h(\tfrac{u}{|u|_h},\tfrac{v}{F(v)})^2)K^h(v,u)\]
does not depend on $u$ for any $u\in T_{\pi(v)}S$ such that $g_v(v,u)=0$.
\end{corollary}
\begin{corollary}\label{flatclass}
Let $(\V, F)$ be a Randers-Minkowski space with Zermelo data $ (h,W)$. Then any $h$-flat hypersurface $S$ of $(\V,h)$ has $F$-scalar flag curvature and its Zermelo metric is conformally flat.
\end{corollary}
\begin{corollary}\label{indicatrix}
Let $(\V, F)$ be a Randers-Minkowski space with Zermelo data $ (h,W)$. The flag curvature of the indicatrix $\Sigma$ at $x\in\Sigma$ is given by
\begin{align}
\nonumber K_v(u)=&\frac{1}{1-h(N_{x},W)^2}\left[\vphantom{\frac{a}{F(a)}}h(\tfrac{v}{F(v)},\tfrac{v}{F(v)})-h(\tfrac{\tilde u}{|\tilde u|_h},\tfrac{v}{F(v)})^2\right.\\
\nonumber&\hspace{-0.8cm} \left.+ \frac{h(v,v)}{2F(v)^3(1-h(N_x,W)^2)}( -h(v,W)(1+3h(N_x,W)^2)+\frac{3}{2F(v)}h(N_x,W)^{2}h(v,v)\right],\label{curvind}
\end{align}
where $v\in TS\setminus \bf 0$, $u\in T_xS$, $\tilde u=u-\frac{1}{\phi(v)}h(v/F(v)-W,u)v/F(v)$ and $N_x=-(x-W)$.
\end{corollary}
\begin{proof}
First observe that the second fundamental form of $\Sigma$ with respect to $h$ is the restriction of $h$ to $\Sigma$ whenever one chooses as the normal vector $N_x=-(x-W)$. Then $\bar\nabla\sigma'=0$, since $\bar\nabla$ is the induced connection, which in Riemannian geometry turns out to be the Levi-Civita connection of the induced metric on $S$. Moreover,
$\sigma'(v,W^{\top_{\pi(v)}})=h(v,W^{\top_{\pi(v)}})=h(v,W)$ as $v$ is tangent to $S$ and using \eqref{Wsigp}, we deduce that
\begin{align}
\nonumber h(v,\WSigP)&=h(v,W)-(-1+\frac{\phi(v)}{1-h(N_x,W)^2})h(\tfrac{v}{F(v)}-W,v)\\&=
h(v,W)\frac{\phi(v)}{1-h(N_x,W)^2}-\frac{1}{F(v)}h(v,v)\left(-1+\frac{\phi(v)}{1-h(N_x,W)^2}\right).
\end{align}
With all this information, and taking into account that the $h$-Riemannian curvature of $\Sigma$ is equal to $1$, \eqref{curvind} follows from \eqref{Kflagformula2}.
\end{proof}
 \begin{remark}\label{pseudoRanders}
Observe that for the sake of simplicity we have written this section for classical Randers metrics, but all the computations, up to some sign, hold for pseudo-Randers-Kropina metrics, which can be characterized by its Zermelo data $(h,W)$, being $h$ a non-degenerate scalar product and $W$ an arbitrary constant vector with no restrictions of norm (see \cite[Section 2.3]{JavVit18}). Several observations are in order:
\begin{enumerate}[(i)]
\item Kropina norms are obtained when $h$ is positive definite and $h(W,W)=1$, and all the submanifolds will be of Kropina type (this can be obtained as a consequence of the generalization of Proposition \ref{ZermSub}).   \item When $h(W,W)>1$, the situation is more complex, as there are two metrics. This is called a wind Riemannian structure in \cite{CJS14}. Our results can be applied to both metrics and an application of Proposition \ref{ZermSub} shows that the induced metric is always a wind Riemannian structure (neither Randers nor Kropina). For a classificacion of wind Riemannian structures with constant flag curvature see \cite{JavSan17}.
\item  In any case, it is also true that the $h$-flat hypersurfaces have $F$-scalar flag curvature and that we can obtain families of Kropina metrics with scalar flag curvature (see \cite{Xia13} for further results on Kropina metrics with scalar flag curvature and \cite{YoOk07,YoSa14} for a classificacion of Kropina metrics with constant flag curvature).
\end{enumerate}
\end{remark} 
\section*{Acknowledgments}
The authors warmly acknowledge useful discussions about the classification of isometric immersions in the Euclidean space of manifolds of constant sectional curvature with Professors Marcos Dajczer (Impa, Rio de Janeiro, Brazil) and Ruy Tojeiro (USP, S\~ao Carlos, Brazil), and Professors J\'ozsef Szilasi (University of Debrecen, Hungary), Nicoleta Voicu (Transilvania University, Brasov, Romania) and Ioan Bucataru (Alexandre Ioan Cuza University, Iasi, Romania) for pointing out some references.

\end{document}